\documentclass[pdflatex,sn-mathphys-num]{sn-jnl}

\usepackage{graphicx}%
\usepackage{multirow}%
\usepackage{amsmath,amssymb,amsfonts}%
\usepackage{amsthm}%
\usepackage{mathrsfs}%
\usepackage[title]{appendix}%
\usepackage{xcolor}%
\usepackage{textcomp}%
\usepackage{manyfoot}%
\usepackage{booktabs}%
\usepackage{algorithm}%
\usepackage{algorithmicx}%
\usepackage{algpseudocode}%
\usepackage{listings}%
\usepackage{graphicx}
\usepackage{subfigure}
\usepackage{hyperref}
\usepackage{float}

\theoremstyle{thmstyleone}%
\newtheorem{theorem}{Theorem}
\newtheorem{lemma}{Lemma}
\theoremstyle{thmstyletwo}%
\newtheorem{remark}{\bf Remark}[section]%
\renewcommand{\proofname}{Proof.}

\theoremstyle{thmstylethree}%

\UseRawInputEncoding
\begin{document}

\title[Structure preserving schemes for the phase field crystal model]{Stability and convergence analysis of unconditionally original energy dissipative implicit-explicit Runge--Kutta methods for the phase field crystal models without Lipschitz assumptions.}


\author[1]{\fnm{Xiaoli} \sur{Li}}\email{xiaolimath@sdu.edu.cn}

\author[2]{\fnm{Kaiyi} \sur{Niu}}\email{202411892@mail.sdu.edu.cn}

\author*[3]{\fnm{Jiang} \sur{Yang}}\email{yangj7@sustech.edu.cn}

\affil[1]{\orgdiv{School of Mathematics and State Key Laboratory of Cryptography and Digital Economy Security}, \orgname{Shandong University}, \orgaddress{\street{Jinan}, \city{Shandong}, \postcode{250100}, \country{P.R. China}}}

\affil[2]{\orgdiv{School of Mathematics}, \orgname{Shandong University}, \orgaddress{\street{Jinan}, \city{Shandong}, \postcode{250100}, \country{P.R. China}}}

\affil[3]{\orgdiv{Department of Mathematics, SUSTech International Center for Mathematics \& National Center for Applied Mathematics Shenzhen (NCAMS), Guangdong Provincial Key Laboratory of Computational Science and Material Design}, \orgname{Southern University of Science and Technology}, \orgaddress{\street{Shenzhen}, \city{Guangdong}, \postcode{518000}, \country{P.R. China}}}


\abstract{The phase field crystal (PFC) method is an efficient technique for simulating the evolution of crystalline microstructures at atomistic length scales and diffusive time scales. Due to the high-order derivatives (sixth-order) and the strongly nonlinear term (locally Lipschitz), developing high-order stable schemes and establishing corresponding error estimates is particularly challenging. In this study, we first establish a general framework for high-order implicit-explicit (IMEX) Runge--Kutta methods that preserves the original energy dissipation for {\bf auxiliary} models with globally Lipschitz truncations on the nonlinear term. By employing the Sobolev embedding theorem and Cauchy's interlace theorem, we demonstrate that the solutions of the auxiliary models are identical to the solutions of the original models without the globally Lipschitz property, provided that the free energy of the initial value is well-defined. Furthermore, we rigorously prove the uniform boundedness of the solution in the $L^{\infty}$ norm and unconditional global-in-time stability. This allows for a straightforward framework to derive optimal arbitrarily high-order $L^{\infty}$ error estimate without relying on the Lipschitz assumption. In particular, compared to existing literature, the argument for error estimation is presented in a much more simplified and elegant manner, without imposing any constraints on time-step size or mesh grid size. In fact, the reported framework, built upon the truncated auxiliary problem for the original model, can be directly extended to a wide range of gradient flows, including Allen--Cahn equations, nonlocal PFC models, and epitaxial thin film growth equations, providing unconditional energy dissipation without enforcing Lipschitz continuity. Finally, we present numerical examples to validate our analytical results and demonstrate the effectiveness of capturing long-time dynamics.}




\maketitle

\section{Introduction}\label{sec1}

The phase field crystal (PFC) model, which was developed in \cite{elder2004modeling,elder2002modeling} by Elder, offers an efficient framework to simulate the evolution of crystalline microstructures on atomic length scales over diffusive time scales. In this model, the phase field variable is introduced to characterize the transition from the liquid phase to the crystalline phase. The PFC equation model can be regarded as the $H^{-1}$ gradient flow of the Swift-Hohenberg type \cite{swift1977hydrodynamic}
\begin{align}\label{eq1.1}
	\mathcal{E}\left(\phi\right)=\int_{\Omega } \left(\frac{1}{4}\phi^4 -\frac{\varepsilon }{2}\phi^2 +\frac{1}{2}{\phi\left(I+\Delta \right)}^2 \phi\right)\mathrm{dx},
\end{align}
where $\Omega \in \mathbb{R}^d\left(d=2,3 \right)$ is a bounded domain, $\phi:\Omega\rightarrow\mathbb{R}$ is the  dimensionless density field, and $\epsilon \in(0,1)$
is a constant proportional to dimensionless undercooling. Hence a wide range of problems have been studied by the PFC model, such as the dislocations motion, boundary structure, defects in materials, crystallization phenomena, and many others; see \cite{asadi2015review, emmerich2012phase, trautt2012coupled, provatas2007using, stefanovic2006phase}.

In this work, we consider the PFC model in the following form:
\begin{align}\label{eq1.2}
	&\frac{\partial \phi}{\partial t}=\Delta \mu,\;\;\;x\in \Omega ,\;t\in\left( 0,T\right],\\
	&\mu={\left(I+\Delta \right)}^2 \phi+f\left(\phi\right),\;\;\;x\in \Omega ,\;t\in\left( 0,T\right], \notag \\
	&\phi \left(x,0\right)=\phi^0 \left(x\right),\;\;\;x\in \Omega , \notag
\end{align}
where $T$ is the final time, $\mu$ is the chemical potential, $\phi\left(x,t\right)\in \Omega \times \left(0,T\right]$, $\phi_0$ is the initial data and $f(\phi)=\phi^3-\varepsilon \phi$. For simplicity, the model above is equipped with periodic boundary conditions or the homogeneous Neumann boundary condition $\frac{\partial \phi}{\partial n}=\frac{\partial \Delta \phi}{\partial n}=\frac{\partial \mu}{\partial n}=0$. As an $H^{-1}$ gradient flow with respect to the energy functional $\eqref{eq1.1}$, the PFC model satisfies the energy dissipation law, i.e. 
$$\frac{d}{dt}\mathcal{E}\left(\phi\right)=-\left\|\nabla \mu\right\|^2_{L^2}\leq0.$$
In addition, it is easy to see that this model is also mass conservative represented by 
$$\frac{d}{dt}\int_{\Omega }\phi \;\mathrm{dx}=0.$$
A closely related model is the Swift-Hohenberg equation \cite{lee2019energy}, which is an $L^2$-gradient flow for the functional free energy \eqref{eq1.1}. This equation dissipates this energy but violates mass conservation.

As a sixth-order nonlinear partial differential equation with a long time scale, the PFC equation cannot generally be solved analytically. Consequently, efficient and structure-preserving numerical methods are desirable and many works have been reported for the PFC equation. A widely employed numerical approach is based on the convex splitting technique of the energy functional. This idea was first proposed by Eyre in \cite{eyre1998unconditionally} for the Cahn-Hilliard equation, where the functional energy admits convex and concave decomposition. The convex component is treated implicitly, while the concave part is treated explicitly \cite{glasner2016improving}. There are generally two different ways of decomposing the energy
\begin{itemize}
	\item convex splitting (CS1)
	\begin{align*}
		E_c(u)=\int_{\Omega}\left( \frac{1}{4}u^4+\frac{1}{2}u\left(I+\Delta \right)^2u \right) \text{dx},\;\;\;E_e(u)=-\frac{\epsilon}{2}\int_{\Omega}u^2\;\mathrm{dx};
	\end{align*}
	\item convex splitting (CS2)
	\begin{align*}
		E_c(u)=\int_{\Omega}\left( \frac{1}{4}u^4+\frac{1-\epsilon}{2}u^2+\frac{1}{2}\left|\Delta u \right|^2 \right)\text{dx},\;\;\;E_e(u)\int_{\Omega}u\Delta u\;\mathrm{dx}. 
	\end{align*}
\end{itemize}
The first one was proposed in \cite{wise2009energy} and the corresponding first- and second-order schemes were established and analyzed in \cite{dong2018convergence,hu2009stable,wise2009energy}. The other one was first considered in \cite{vignal2015energy}, and then the first- and second-order schemes were developed in \cite{li2018second,shin2016first,vignal2015energy}. All these convex splitting techniques require an implicit treatment of the cubic term in $f(\phi)$ due to the convexity structure, which means that nonlinear iterations exist in implementation.

To obtain the linear and original energy stable scheme, the stabilization technique \cite{wise2009energy,xu2006stability} has been widely employed for gradient flows. The constructed scheme in \cite{wise2009energy} can be proven to satisfy
energy stability, provided that the stabilizing constant is sufficiently large, which is a requirement that depends on the uniform bound of the unknown numerical solutions. However, increasing the stabilizer coefficient significantly amplifies the truncation error, as extensively verified in numerical simulations. Moreover, a critical challenge arises in maintaining unconditional energy dissipation when applying stabilization techniques to high-order schemes in time \cite{akrivis2022error,li2020arbitrarily}. In addition, a class of methods based on the auxiliary function approach has been proposed for the PFC-related models and general gradient flows, including the invariant energy quadratization (IEQ) method \cite{yang2017linearly} and the scalar auxiliary variable (SAV) approach \cite{li2020stability,li2022efficient,shen2018scalar,shen2019new,akrivis2019energy}, which are all founded on the modified energy dissipation. To construct the linear, high-order and original energy dissipative schemes, a series of Runge--Kutta methods have been researched for PFC model and other gradient flows in recent years. For instance, by combining the energy quadratization (EQ) technique and a specific class of Runge--Kutta methods, a class of arbitrarily high-order algorithms for gradient flow models \cite{gong2019energy} was developed that unconditionally dissipates energy. Furthermore, as demonstrated in \cite{shin2017unconditionally}, the convex splitting Runge--Kutta (CSRK) scheme provides a robust and unified framework for the numerical solution of gradient flows. This approach ensures unconditional energy stability and can achieve high-order temporal accuracy, making it a highly effective method for such problems.

In fact, the PFC model is primarily employed for long-time simulations, which implies that the stability analyses should try to avoid relying on local in time properties. Theoretically, as the final simulation time increases, the reliability of the numerical solutions always diminishes in comparison to the exact solutions due to an exponential growth constant in convergence analysis. Particularly for gradient flow models, in the absence of an enforced Lipschitz condition, the boundedness of the numerical solution and stability results should be established by leveraging convergence results in conjunction with induction method, which typically guarantee only local-in-time stability. It has been widely recognized that developing long-time stable and efficient numerical schemes is crucial for designing algorithms capable of accurately capturing long-term statistical behavior \cite{coleman2024efficient,wang2010approximation,wang2012efficient}. Recently, Li et al. \cite{li2024globalintimeenergystabilityanalysis} presented a global-in-time energy estimate for the second-order accurate exponential time differencing Runge--Kutta (ETDRK2) numerical scheme for the PFC model, where the energy dissipation property is valid for any final time. Based on the stabilization technique from the existing work \cite{Qiao2024pfc} and the $H^2$ term existing in energy functional, the numerical solution generated by the
ETDRK2 scheme satisfies the global-in-time, but not unconditional energy stability since the artificial parameter depends only on the initial energy and the domain with an $O(1)$ constraint on time step. A second-order accurate global-in-time energy stable IMEX-RK method has also been proposed in \cite{doi:10.1137/24M1637623}. The analysis of the global-in-time energy stable RK method with mild constraints for the time steps in \cite{li2024globalintimeenergystabilityanalysis,doi:10.1137/24M1637623} is stage-by stage, which is rigorous but relatively lengthy to extend RK method to high-order schemes.

For the general gradient flows, Fu et al. \cite{fu2024energydiminishingimplicitexplicitrungekutta} constructed a class of high-order IMEX-RK methods for gradient flows under the assumption of the global or local Lipschitz continuity of the nonlinear term. Benefited from the stabilization technique, the original energy dissipation property without any restrictions on the time step can be proved because the stabilizers only depend on the Butcher notation of IMEX-RK. Also a simple framework that can determine whether an IMEX-RK method is feasible to preserve the original energy dissipation property or not is established. However, this result depends the Lipschitz condition for the nonlinear term. This relatively strict constraint motivates us to explore alternative approaches that circumvent the need for the Lipschitz condition while still guaranteeing provable unconditional energy dissipation. 

In this work, inspired by \cite{fu2024energydiminishingimplicitexplicitrungekutta}, we construct a class of linear and arbitrary high-order IMEX-RK schemes with stabilizers for PFC model. The proposed scheme can be demonstrated to achieve unconditional energy dissipation without imposing the Lipschitz condition. Our primary contributions are summarized as follows:
\begin{itemize}
	\item[$\bullet$] To circumvent the necessity of imposing the Lipschitz condition on the nonlinear term, we initially establish Lemma \ref{lemma3.1}, which demonstrates that the uniform boundedness can be achieved through energy dissipation argument. This approach leverages the existence of a sixth-order dissipation term to ensure the desired boundedness.
	
	\item[$\bullet$] We propose an auxiliary problem by truncating the nonlinear term to exhibit quadratic growth beyond the maximum norm of the numerical solutions, denoted by $M$. This approach enables a straightforward proof of unconditional energy dissipation and uniform boundedness. Then the equivalence between the original and auxiliary formulations has been rigorously established by using the Cauchy interlacing theorem. 
	
	\item[$\bullet$] The proposed scheme can be demonstrated to achieve unconditional energy dissipation, which is an essential improvement comparing with the unifrom-in-time energy stability with conditional time step $O(1)$ as reported in \cite{li2024globalintimeenergystabilityanalysis}. Furthermore, an error estimate in the discrete $L^{\infty}$ norm  is derived by leveraging the uniform boundedness of the numerical solution and Sobolev embedding theorem.
\end{itemize}
Specifically, to ensure the uniform boundedness of the numerical solutions, particularly at the intermediate stages of the IMEX-RK scheme, we verify that the chosen stabilizers are also effective for these intermediate functions as well. This verification relies on eigenvalue interlacing properties, specifically the Cauchy interlacing theorem, which relates symmetric matrices to their principal submatrices. These interlacing relationships establish that the maximum norm of the numerical solutions remains bounded by the constant $M$. Then the truncation of the nonlinear term does not compromise the fundamental properties of the numerical solutions.
Consequently, the equivalence between the original and auxiliary formulations has been rigorously established, thereby removing the requirement for the global Lipschitz continuity of the nonlinear term. Our analysis is based on the high-order temporal discretization in this work, the framework to achieve unconditional uniform boundedness,  followed by deriving $L^{\infty}$-error estimates using energy methods, is broadly feasible to match with any spatial discretization methods, such as finite element,  finite difference and spectral methods.

This paper is organized as follows. In Section 2, we introduce some preliminaries concerning splitting techniques and IMEX-RK schemes. We establish the energy dissipation and uniform boundedness in Section 3. And we give the $L^{\infty}$-convergence for the constructed scheme in section 4. Section 5 offers several numerical examples. Some concluding remarks are given in the final section.

\section{Energy preserving IMEX-RK methods for PFC model with stabilization technique}
In this section, we develop a linear, high-order, unconditionally energy stable scheme by using IMEX-RK methods to solve PFC equation \eqref{eq1.1} with stabilizers. And we will also show that the constructed scheme is unconditionally energy dissipative with appropriate stabilizers.
\subsection{A stabilization technique}\label{subsection2.1}
First, we introduce a stabilizer into the linear part of $\mu$ in \eqref{eq1.2}, so we derive the following equivalent equation
\begin{align}\label{eq2.1}
	\phi_t =\Delta \left({\left(\Delta +I\right)}^2 \phi+a\phi+\phi^3 -\left(\varepsilon +a\right)\phi\right),
\end{align}
where $a \in \mathbb{R}$ is a given constant and its range will be presented in convergence analysis. So the nonlinear function $f$ and energy functional $\mathcal{E}\left(\phi\right)$ also need to be transformed into follows
\begin{gather}
	f\left(\phi\right)=\phi^3 -\left(\varepsilon +a\right)\phi,~~F'=f,\label{eq2.2}\\
	\mathcal{E}\left(\phi\right)=\int_{\Omega } \left({\frac{1}{2}\phi\left(\Delta +I\right)^2} \phi+\frac{a}{2}\phi^2 +F\left(\phi\right)\right)\mathrm{d}\mathrm{x}.\label{eq2.3}
\end{gather}
For simplicity, we write the linear operator of energy \eqref{eq2.3} as
\begin{align}
	P_a ={\left(\Delta +I\right)}^2+aI.
	\nonumber
\end{align}

Consider the splitting of the energy functional $\mathcal{E}\left(\phi\right)=\mathcal{E}_l \left(\phi\right)-\mathcal{E}_n \left(\phi\right)$ with
\begin{gather}\label{eq2.4}
	\mathcal{E}_l \left(\phi\right)=\int_{\Omega } \frac{1}{2}\left(\phi\left(\Delta +I\right)^2 \phi+a\phi^2\right)+\frac{\alpha }{2}\left\lbrack {\phi\left(\Delta +I\right)}^2 \phi+a\phi^2 \right\rbrack +\frac{\beta }{2}\phi^2\; \mathrm{dx},\\
	\mathcal{E}_n \left(\phi\right)=\int_{\Omega } -F\left(\phi\right)+\frac{\alpha }{2}\left\lbrack {\phi\left(\Delta +I\right)}^2 \phi+a\phi^2 \right\rbrack +\frac{\beta }{2}\phi^2\;\mathrm{dx},
\end{gather}
where $\alpha, \beta\in\mathbb{R}$ are stabilizers and their ranges will be presented in energy decreasing analysis. Now based on splitting of energy, we obtain the following equivalent equation again
\begin{align}\label{eq2.6}
	\phi_t =\mathcal{L}\phi+N\left(\phi\right),
\end{align}
where
\begin{gather}
	\mathcal{L}=\Delta \left(\left(1+\alpha \right)P_a +\beta I\right),\nonumber\\
	N\left(\phi\right)=-\Delta \left(-f\left(\phi\right)+\alpha P_a\phi+\beta \phi\right),
	\nonumber
\end{gather}
\begin{remark}
	The stabilizer $a\phi$ in \eqref{eq2.1} is essential to derive the $L^{\infty}$-convergence in error estimate, since by using this stabilizer with any $a>0$, the terms
	$\left\|e_j^n \right\|_{L^{\infty}} $ and $\left\|\nabla e^n_j \right\|_{L^2}$ can be governed by $\left(e_j^n ,P_a e_j^n \right)$, as stated in \eqref{eq4.8}-\eqref{eq4.10}, then we can derive the uniform boundedness of numerical solution. On the other hand, $\alpha, \beta$ in \eqref{eq2.4} are selected to ensure the unconditional energy dissipation. Due to their different roles, we illustrate them respectively.
\end{remark}

\subsection{Implicit-Explicit Runge--Kutta method}
For the linear term $\mathcal{L}\phi$ in \eqref{eq2.6}, we use an $s$-stage diagonally implicit Runge--Kutta (DIRK) method with coefficient matrix $A={\left(a_{ij} \right)}_{s\times s} \in {\mathbb{R}}^{s\times s}$, $c,b \in {\mathbb{R}}^s$ as their butcher table. For the nonlinear term $N\left(\phi\right)$, we consider an $s$-stage explicit method with coefficient matrix $\hat{A}={\left(\hat{a}_{ij} \right)}_{s\times s} \in {\mathbb{R}}^{s\times s}$, $\hat{c},\hat{b} \in {\mathbb{R}}^s$. Then we can construct the linear implicit method for the PFC model. Furthermore, we choose IMEX-RK method with following conditions
\begin{itemize}
	\item[$\bullet$]$c_i=\sum_{j=1}^s a_{ij}=\hat{c}_i=\sum_{j=1}^s {\hat{a} }_{ij}$.
	
	\item[$\bullet$]$b_j=a_{sj}$, $\hat{b}_j=\hat{a}_{sj}$, $j=1,...,s$, implying that the implicit method is stiffly accurate.
	
	\item[$\bullet$]Matrix $\hat{A}$ is invertible.
\end{itemize}

To emphasize the relationship between implicit and explicit Runge--Kutta methods, we incorporate zeros into their Butcher tables. Thus, the IMEX-RK can be determined by the following Butcher notation
\begin{align}\label{eq2.7}
	\begin{array}{c|ccccc}
		0 & 0 & 0 & ... & ... & 0 \\
		c_1 & 0 & a_{11} & 0 & ... & 0 \\
		c_2 & 0 & a_{21} & a_{22} & ... & 0 \\
		... & 0 & ... & ... & ...& ... \\
		c_s & 0 & a_{s1} & a_{s2} & ... & a_{ss} \\
		\hline
		& 0 & b_1 & b_2 & ... & b_s
	\end{array}
	\qquad
	\begin{array}{c|ccccc}
		0 & 0 & 0 & 0 & ... & 0 \\
		\hat{c}_1 & \hat{a}_{11} & 0 & 0 & ... & 0 \\
		\hat{c}_2 & \hat{a}_{21} & \hat{a}_{22} & 0 & ... & 0 \\
		... & ... & ... & ... & ... & 0 \\
		\hat{c}_s & \hat{a}_{s1} & \hat{a}_{s2} & ... & \hat{a}_{ss} & 0 \\
		\hline
		& \hat{b}_1 & \hat{b}_2 & ... & \hat{b}_s & 0
	\end{array}
\end{align}

Applying the IMEX-RK method \eqref{eq2.7} to model \eqref{eq2.6}, we derive the following system (solving $\phi_{n+1}$ form $\phi_n$)
\begin{align}\label{eq2.8}
	\left\lbrace \begin{array}{ll}
		u^n_0 =\phi^n, & \\
		u^n_i =u^n_0 +\tau\left(\sum_{j=1}^i a_{ij} {Lu}^n_j +\sum_{j=1}^i \hat{a}_{ij} N\left(u^n_{j-1} \right)\right), & 1\le i\le s,\\
		\phi^{n+1} =u^n_s. & 
	\end{array}\right.
\end{align}

\section{Energy dissipation and uniform boundedness}
As mentioned before, benefited from the stabilization technique, Fu et al. \cite{fu2024energydiminishingimplicitexplicitrungekutta} constructed a class of IMEX-RK methods that are the first provably high-order and linearly unconditionally energy-stable methods for gradient flows with Lipschitz assumption.  However, the Lipschitz continuity condition on the nonlinear term is often a restrictive assumption in many analytical framework.  Therefore, developing approaches focusing on the $H^2$ term in energy functional of PFC model to circumvent this constraint could offer significant theoretical and practical advantages.

In this section, we establish the relationship between energy dissipation law and the uniform boundedness of numerical solution. We demonstrate these properties by replacing the PFC model with an auxiliary problem that intrinsically features a Lipschitz continuous nonlinearity. The core idea is that if we can prove the auxiliary problem and the original PFC model yield identical numerical solution, then these solutions inherently share the same qualitative properties, including energy dissipation and boundedness.

We first present a lemma demonstrating that the energy dissipation ensures the uniform boundedness of the solution. 
This lemma is crucial for demonstrating that, provided the numerical solution at any given time step preserves the energy dissipation property relative to the initial data $\phi_0$, the corresponding $L^{\infty}$ norm can subsequently be controlled by a bound $M(\phi_0,\Omega)$. This is particularly significant because the mass conservation has already been inherently enforced.

\medskip
\begin{lemma}\label{lemma3.1}
	Suppose that $u$, $v:\bar{\Omega}\rightarrow\mathbb{R}$ are functions satisfying $\int_{\Omega}u\;dx=\int_{\Omega}v\;dx$ and $\mathcal{E}(u)\leq\mathcal{E}(v)$, equipped with either the periodic boundary condition or the homogeneous Neumann boundary condition. Then $\left\|u \right\|_{L^\infty}$ can be governed only by $v$ and $\Omega$.
\end{lemma}
\begin{proof}
	Above all, we introduce some useful inequalities:
	\begin{gather}
		{\left\|\nabla u\right\|}_{L^2 }^2 \le p{\left\|u\right\|}_{L^2 }^2 +q{\left\|\Delta u\right\|}_{L^2 }^2,\label{eq3.1}\\
		{\left\|u\right\|}_{L^{\infty } } \le C_{\Omega}\sqrt{{\left\|u\right\|}_{L^2 }^2 +{\left\|\nabla u\right\|}_{L^2 }^2 +{\left\|\Delta u\right\|}_{L^2 }^2 },\label{eq3.2}
	\end{gather}
	and Poincar\'{e}-Wirtinger inequality:
	\begin{align}\label{eq3.3}
		\left\|u-\bar{u} \right\|_{L^2} \leq C_p\left\|\nabla u \right\|_{L^2},
	\end{align}
	where $C_{\Omega}$ and $C_p$ are constants that only dependent $\Omega$, $\bar{u}=\frac{1}{\left|\Omega \right| }\int_{\Omega}u\;dx$, and $2\sqrt{pq}=1$.

	Now we first prove that $\left\|u \right\|_{L^\infty}$ can be governed by its energy, domain $\Omega$ and the mean of itself on $\Omega$. We start with energy $\mathcal{E}(u)$:
	\begin{align}\label{eq3.4}
		\mathcal{E}\left(u\right)&=\frac{1}{4}(u^4,1)+\frac{1-\epsilon}{2}\left\|u \right\|^2_{L^2}-\left\|\nabla u \right\|^2_{L^2}+\frac{1}{2}\left\|\Delta u \right\|^2_{L^2}\\
		&\geq \frac{1}{4}(u^4,1)-\frac{1+\epsilon}{2}\left\|u \right\|^2_{L^2}+\frac{1}{4}\left\|\Delta u \right\|^2_{L^2}\notag\\
		&\geq \left( \frac{1}{4}-\frac{(1+\epsilon)^2}{16}\right) (u^4,1)-\left|\Omega \right|+\frac{1}{4}\left\|\Delta u \right\|^2_{L^2},\notag
	\end{align}
	where $\left|\Omega \right|=\int_{\Omega} 1\;dx $. Considering that $\epsilon\in(0,1)$, we have
	\begin{align}\label{eq3.5}
		\left\|\Delta u \right\|_{L^2}\leq 2\left(\mathcal{E}(u)+\left|\Omega \right| \right)^{\frac{1}{2}}=M_1\left(u,\Omega \right) .
	\end{align}
	Here we also need to obtain a bound of $\left\| u \right\|_{L^2}$ and $\left\|\nabla u \right\|_{L^2}$ if we want to use inequality \eqref{eq3.2}. Obtained by the Poincar\'{e}-Wirtinger Inequality, we have
	\begin{align}\label{eq3.6}
		\left\|u \right\|_{L^2} & \leq \left\|u-\bar{u}\right\|_{L^2} + \left\|\bar{u}\right\|_{L^2}\\
		&\leq  C_p\left\|\nabla u \right\|_{L^2} + \left\|\bar{u}\right\|_{L^2}.\notag
	\end{align}

	Utilizing inequality \eqref{eq3.1}, we can derive the following inequality
	\begin{align}\label{eq3.7}
		{\left\|\nabla u\right\|}_{L^2 } \le \frac{\epsilon}{2}{\left\|u\right\|}_{L^2 } +\frac{1}{2\epsilon}{\left\|\Delta u\right\|}_{L^2 }.
	\end{align}
	Selecting suitable $\epsilon$, satisfying that $\frac{\epsilon}{2}C_p \leq \frac{1}{2}$, we can obtain an estimate on $\left\|u \right\|_{L^2}$:
	\begin{align}\label{eq3.8}
		\left\|u \right\|_{L^2} \leq \frac{C_p}{\epsilon} \left\|\Delta u\right\|_{L^2 }+2\left\|\bar{u}\right\|_{L^2}=M_2\left(u,\Omega \right) ,
	\end{align}
	and the estimate on $\left\|\nabla u\right\|_{L^2}$ can be inferred form \eqref{eq3.8}:
	\begin{align}\label{eq3.9}
		{\left\|\nabla u\right\|}_{L^2 } \leq \epsilon\left\|\bar{u}\right\|_{L^2} +\left( \frac{C_p}{2}+\frac{1}{2\epsilon}\right) {\left\|\Delta u\right\|}_{L^2 }=M_3\left(u,\Omega \right) .
	\end{align}

	It is thus clear that $\left\|u\right\|_{L^{\infty}}\leq M(u,\Omega)$, where $M(u,\Omega):=C_{\Omega}\sqrt{M_1^2+M_2^2+M_3^2}$ for simplicity. Following the facts that $\int_{\Omega}u\;dx=\int_{\Omega}v\;dx$ and $\mathcal{E}(u)\leq\mathcal{E}(v)$, we have $M(u,\Omega)\leq M(v,\Omega)$. So $\left\|u \right\|_{L^\infty}$ can be governed by $v$ and $\Omega$.
\end{proof}

\textbf{\textit{The auxiliary problem}}:  let $M_0=M(\phi_0,\Omega)$, then the auxiliary problem with the help of lemma \ref{lemma3.1} can be proposed in the following form:
\begin{align}\label{eq3.10}
	\left\lbrace \begin{array}{ll}
		\tilde{u}^n_0 =\tilde{\phi}^n, & \\
		\tilde{u}^n_i =\tilde{u}^n_0 +\tau\left(\sum_{j=1}^i a_{ij} {\mathcal{L}\tilde{u}}^n_j +\sum_{j=1}^i \hat{a}_{ij} \tilde{N}\left(\tilde{u}^n_{j-1} \right)\right), & 1\le i\le s, \\
		\tilde{\phi}^{n+1} =\tilde{u}^n_s, & 
	\end{array}\right.
\end{align}
where 
\begin{align}
	\tilde{\phi}^0=\phi^0,\quad\tilde{N}\left(\phi\right)&=-\Delta \left(-\tilde{f}\left(\phi\right)+\alpha P_a\phi+\beta \phi\right),
\end{align}
and
\begin{align}\label{eq3.12}
	\tilde{f} \left(\phi \right)= \left\lbrace \begin{array}{ll}
		\left(3M_0^2 -\left(\epsilon +a\right)\right)\phi -2M_0^3 , & \phi >M_0,\\
		\phi^3 -\left(\epsilon +a\right)\phi , & \phi \in \left\lbrack -M_0,M_0\right\rbrack\\
		\left(3M_0^2 -\left(\epsilon +a\right)\right)\phi +2M_0^3 , & \phi <-M_0\ldotp 
	\end{array}\right..
\end{align}

Actually, the construction of \eqref{eq3.12} can be regarded as a technique that truncates the function $F$ to exhibit quadratic growth at infinity, then it is trivial for $\tilde{f}$ to be Lipschitz bounded with Lipschitz constant $L$. More precisely, inspired by \cite{shen2010numerical},  we replace $F(\phi)=\frac{1}{4}\phi^4-\frac{\epsilon+a}{2}\phi^2$ by:
\begin{align}\label{eq3.13}
	\tilde{F} \left(\phi \right)=\left\lbrace \begin{array}{ll}
		\frac{3M_0^2 -\left(\epsilon +a\right)}{2}\phi^2 -2M_0^3 \phi +\frac{3}{4}M_0^4 , &\ \phi >M_0,\\
		\frac{1}{4}\phi^4 -\frac{\left(\epsilon +a\right)}{2}\phi^2 , &\ \phi \in \left\lbrack -M_0,M_0\right\rbrack \\
		\frac{3M_0^2 -\left(\epsilon +a\right)}{2}\phi^2 +2M_0^3 \phi +\frac{3}{4}M_0^4 , &\ \phi <-M_0\ldotp 
	\end{array}\right.,
\end{align}
and declare that $\tilde{f} \left(\phi \right)=\tilde{F}' \left(\phi \right)$.

\medskip
\begin{lemma}
	The auxiliary problem \eqref{eq3.10} and the original system \eqref{eq2.8} share the same solution provided that
	$\mathcal{E}\left(\tilde{u}^n_i \right) \leq\mathcal{E}(\tilde{\phi}^0 )$ for every $i$ and $n$, i.e., the solutions of \eqref{eq3.10} and \eqref{eq2.8} at each time step, including all intermediate solutions, are energy stable with respect to the initial condition.
\end{lemma}

\begin{proof}
	It is clear that, system \eqref{eq3.10} and \eqref{eq2.8} have the same solution when $\left\| \tilde{u}^n_i\right\|_{L^{\infty}}\leq M_0$, for every $i$ and $n$. According to lemma \ref{lemma3.1}, we can obtain the uniform boundedness $\left\| \tilde{u}^n_i\right\|_{L^{\infty}}\leq M_0$ if $\mathcal{E}\left(\tilde{u}^n_i \right) \leq\mathcal{E}(\tilde{\phi}^0 )$, which leads to the desired result.
\end{proof}

Using the above lemma, we can deduce that in order to establish the equivalence between the auxiliary problem \eqref{eq3.10} and the original system \eqref{eq2.8}, it is essential to first demonstrate the uniform boundedness of the intermediate solutions lying between $\tilde{u}^n_0$ and $\tilde{u}^n_s$. To facilitate the proofs of energy dissipation and uniform boundedness, we present the following lemma. Inspired by \cite{Cauahyinterlace}, we recall the Cauchy interlace theorem, which elucidates the relationship between the eigenvalues of a real symmetric matrix and those of its principal submatrices. Specifically, it states that the eigenvalues of any principal submatrix are interlaced with those of the original matrix.

\medskip
\begin{lemma}[Cauchy's interlace theorem]\label{lemma3.2}
	Let $A\in \mathbb{R}^{n\times n}$ be a real symmetric matrix with eigenvalues $\lambda_1 \le \lambda_2 \le \cdots \le \lambda_n$, and let $B$ be an $\left(n-1 \right)\times \left(n-1 \right)  $ principal submatrix of $A$ with eigenvalues $\mu_1 \le \mu_2 \le \cdots \le \mu_{n-1}$.\\
	Then
	\begin{align*}
		\lambda_1 \le \mu_1 \le \lambda_2 \le \mu_2 \le \cdots \le \lambda_{n-1} \le \mu_{n-1} \le \lambda_n.
	\end{align*}
	In general, if $B$ be an $m\times m$ principal submatrix of $A$, and $\mu_1 \le \mu_2 \le \cdots \le \mu_{m}$ its eigenvalues, then 
	\begin{align*}
		\lambda_k \le \mu_k \le \lambda_{k+n-m},
	\end{align*}
	for any $k=1,2,...,m$.
\end{lemma}
\begin{proof}
	As noted in \cite{Cauahyinterlace}, proofs of this theorem typically rely on Sylvester's law of inertia and the Courant-Fischer min-max theorem, which can be somewhat intricate. To facilitate a clearer understanding, a simplified framework is presented below.
	
	We set $A= \left[  \begin{matrix}
		B & X^T \\ X & Z
	\end{matrix} \right]$, with the eigenvectors $\left\lbrace x_1,...,x_n \right\rbrace$, and the eigenvectors of $B$ are $\left\lbrace y_1,...,y_n \right\rbrace$, where $x_i$ is the eigenvector of $\lambda_i$ which is linearly independent on others, and so do $B$'s. Define the vector spaces $V=span\left\lbrace x_k,...,x_n\right\rbrace$, $W=span\left\lbrace y_1,...,y_k\right\rbrace$, $U=\left\lbrace\left( \begin{matrix}
		w \\ 0
	\end{matrix}\right)\in \mathbb{R}^n, w\in W \right\rbrace$.

	Clearly, we have $dim(V)=n-k+1$, $dim(W)=dim(U)=k$, then $V\cap U\neq \varnothing$ inasmuch as $dim(V)+dim(U)=n+1>n$. So there is $u\in V\cap U$, also the $u$ can be indicated as $u=\left( \begin{matrix}
		w \\ 0
	\end{matrix}\right)$, for some $w\in W$.

	Then we obtain $u^T Au=[w^T\;0]\left[  \begin{matrix}
		B & X^T \\ X & Z
	\end{matrix} \right]\left[  \begin{matrix}
		w \\ 0
	\end{matrix} \right]=w^T Bw$ and $\dfrac{u^T Au}{u^Tu}=\dfrac{w^T Bw}{w^Tw}$. In light of \emph{Courant-Fischer min-max Theorem}, we have
	\begin{align*}
		\lambda_k=\mathop{min}\limits_{x\in V}\frac{x^T Ax}{x^Tx},\;\mu_k=\mathop{max}\limits_{x\in W}\frac{x^T Bx}{x^Tx}.
	\end{align*}
	Therefore $\dfrac{u^T Au}{u^Tu}\geq \lambda_k$, $\dfrac{w^T Bw}{w^Tw}\leq\mu_k$, so $\lambda_k\leq u_k$.

	On the other side, we may set $V=span\left( x_1,...,x_{k+n-m}\right)$, $W=span\left( y_k,...,y_m\right)$. Hence we can use similar procedure to obtain $\lambda_{k+n-m}\geq \mu_k$.
\end{proof}

Now we present the main theorem for the energy dissipation and uniform boundedness of the auxiliary problem \eqref{eq3.10}.

\medskip
\begin{theorem}\label{the3.3}
	Define matrices $H_0$, $H_1(\beta)$, $H_2(\alpha)$, $Q$, and $\bar{M} \triangleq \frac{1}{2}\left(M+M^{\mathrm{T}} \right)$, $M[m]$ is the principal submatrix of $M$ of order $m$ for any matrix $M\in \mathbb{R}^{n\times n}$, where
	\begin{align}
		&H_0={{\hat{A} }^{-1} E}_L,\label{eq3.14}\\
		&H_1(\beta)=\beta Q-\frac{L}{2}I,\label{eq3.15}\\
		&H_2(\alpha)=\alpha Q-\frac{1}{2}E+{\hat{A} }^{-1} {AE}_L,\label{eq3.16}\\
		&Q=\left({\hat{A} }^{-1} A-I\right)E_L +I.\label{eq3.17}
	\end{align}
	In these matrices, $\alpha$, $\beta\in \mathbb{R}$ are stabilizers, $E\in \mathbb{R}^{s\times s}$ is all $1$ inside, $E_L\in \mathbb{R}^{s\times s}$ is $1$ in lower triangle and other elements are $0$, $I\in \mathbb{R}^{s\times s}$ represents the identity matrix. Furthermore, we define $\lambda_{\mathrm{min}} \left(M\right)$ is the smallest eigenvalue of matrix $M$, and the selection of Butcher notation \eqref{eq2.7} can guarantee that both $\bar{Q}$ and $\bar{H_0}$ are positive-definite. Meanwhile, if $\alpha$ and $\beta$ satisfy the following inequalities:
	\begin{equation}\label{eq3.18}
		\begin{aligned}
			&\alpha \ge \frac{1}{2\lambda_{\mathrm{min}} \left(\bar{Q} \right)}-1,\\
			&\beta \ge \frac{L}{2\lambda_{\mathrm{min}} \left(\bar{Q} \right)},
		\end{aligned}
	\end{equation}
	$\bar{H_1}(\beta)$, $\bar{H_2}(\alpha)$ are both positive-definite.
	Then the IMEX-RK methods \eqref{eq2.8} and \eqref{eq3.10} for the PFC model satisfy the following properties:
	\begin{itemize}
		\item [\romannumeral1)] The IMEX-RK method \eqref{eq3.10} satisfies the unconditional dissipation law:
		\begin{align*}
			&\quad\ \mathcal{E}\left(\tilde{\phi}^{n+1} \right)-\mathcal{E}\left(\tilde{\phi}^n \right)\leq 0.
		\end{align*}
		\item [\romannumeral2)] The IMEX-RK method \eqref{eq3.10} satisfies the uniform boundedness:
		\begin{align*}
			\left\|\tilde{u}^n_i \right\|_{L^{\infty}}\leq M_0,\quad i=0,1,...,s,\; n=0,1,...,\frac{T}{\tau}
		\end{align*}
		\item [\romannumeral3)] The IMEX-RK methods \eqref{eq2.8} and \eqref{eq3.10} share the same numerical solution.
	\end{itemize}
\end{theorem}

\medskip
\begin{proof}
	Let's check out first that how condition \eqref{eq3.18} keeps $\bar{H_1}(\beta)$, $\bar{H_2}(\alpha)$ positive-definite. We find that
	\begin{equation}\label{M}
		\begin{aligned}
			&H_1(\beta)=\beta Q-\frac{L}{2}I,\\
			&H_2 (\alpha)=\left(\alpha+1\right)Q+E_L -I-\frac{1}{2}E,\\
		\end{aligned}
	\end{equation}
	then
	\begin{equation}\label{Mbar}
		\begin{aligned}
			&\bar{H_1}(\beta)=\beta \bar{Q}-\frac{L}{2}I,\\
			&\bar{H_2} (\alpha )=\left(\alpha +1\right)\bar{Q}-\frac{1}{2}I.
		\end{aligned}
	\end{equation}
	It is thus clear that $\bar{H_1}(\beta)$ and $\bar{H_2}(\alpha)$ are positive-definite if $Q$ is positive-definite and $\alpha$, $\beta$ satisfy \eqref{eq3.18}.

	The system of the IMEX-RK method \eqref{eq3.10} has the following matrix form:
	\begin{align}
		\tilde{u}_0^n & =\tilde{\phi}^n,\label{eq3.20}\\
		\left(\begin{array}{c}
			\tilde{u}_1^n \\
			\tilde{u}_2^n \\
			\vdots \\
			\tilde{u}_s^n 
		\end{array}\right)&=\left(\begin{array}{c}
			\tilde{u}_0^n \\
			\tilde{u}_0^n \\
			\vdots \\
			\tilde{u}_0^n 
		\end{array}\right)+\tau\left(A\left(\begin{array}{c}
			{\mathcal{L}\tilde{u}}_1^n \\
			{\mathcal{L}\tilde{u}}_2^n \\
			\vdots \\
			{\mathcal{L}\tilde{u}}_s^n 
		\end{array}\right)+\hat{A} \left(\begin{array}{c}
			\tilde{N}\left(\tilde{u}_0^n \right)\\
			\tilde{N}\left(\tilde{u}_1^n \right)\\
			\vdots \\
			\tilde{N}\left(\tilde{u}_{s-1}^n \right)
		\end{array}\right)\right).\label{eq3.20.1}
	\end{align}
	Following the similar process of \cite[Theorem 3.1.]{fu2024energydiminishingimplicitexplicitrungekutta} and using the Lipschitz continuity of $\tilde{f}$, we derive
	\begin{align}\label{eq3.21}
		&\quad\ \left(\tilde{F}(\tilde{\phi}^{n+1})-\tilde{F}(\tilde{\phi}^n),\textbf{1} \right)=\sum^{s-1}_{i=0}\left(\tilde{F}(\tilde{u}^{n}_{i+1})-\tilde{F}(\tilde{u}^{n}_{i}),\textbf{1} \right)  \\
		&\leq \left(\tilde{u}^n_1-\tilde{u}^n_0,\tilde{u}^n_2-\tilde{u}^n_1,\cdots,\tilde{u}^n_s-\tilde{u}^n_{s-1} \right)\left(\begin{array}{c}
			\tilde{f}(\tilde{u}_0^n) \\
			\tilde{f}(\tilde{u}_1^n) \\
			\vdots \\
			\tilde{f}(\tilde{u}_{s-1}^n) 
		\end{array} \right)
		+\frac{L}{2}\left(\tilde{u}^n_1-\tilde{u}^n_0,\tilde{u}^n_2-\tilde{u}^n_1,\cdots,\tilde{u}^n_s-\tilde{u}^n_{s-1} \right)^2 \notag\\
		&=\tilde{\omega}^T\left(\begin{array}{c}
			\tilde{f}(\tilde{u}_0^n) \\
			\tilde{f}(\tilde{u}_1^n) \\
			\vdots \\
			\tilde{f}(\tilde{u}_{s-1}^n) 
		\end{array} \right)+\frac{L}{2}\tilde{\omega}^2 ,\notag
	\end{align}
	and
	\begin{align}\label{eq3.22}
		&\quad\ \mathcal{E}\left(\tilde{\phi}^{n+1} \right)-\mathcal{E}\left(\tilde{\phi}^n \right)\\
		&\le \frac{L}{2}\tilde{\omega}^2+\frac{1}{2}\tilde{\omega}^T E\left(P_a \tilde{\omega} \right)
		{+\tilde{\omega} }^T \left(\frac{1}{\tau}{{\hat{A} }^{-1} E}_L \left(\Delta^{-1} \tilde{\omega} \right)-{\hat{A} }^{-1} {AE}_L \left(P_a \tilde{\omega} \right)\right)\notag\\
		&\quad-\alpha \tilde{\omega}^T Q\left(P_a \tilde{\omega} \right)-\beta w^T Q\tilde{\omega} \notag\\
		&={-\tilde{\omega} }^T H_2 \left(P_a \tilde{\omega} \right)-\tilde{\omega}^T H_1 \tilde{\omega} +\frac{1}{\tau}\tilde{\omega}^T H_0 \left(\Delta^{-1} \tilde{\omega} \right)\notag\\
		&={-\tilde{\omega} }^T \bar{H_2} \left(P_a \tilde{\omega} \right)-\tilde{\omega}^T \bar{H_1} \tilde{\omega} +\frac{1}{\tau}\tilde{\omega}^T \bar{H_0} \left(\Delta^{-1} \tilde{\omega} \right),\notag
	\end{align}	
	where $\tilde{\omega}^T=\left(\tilde{u}^n_1-\tilde{u}^n_0,\tilde{u}^n_2-\tilde{u}^n_1,\cdots,\tilde{u}^n_s-\tilde{u}^n_{s-1} \right)$, $\textbf{1}$ represents a s-dimensional vector with all entries equal to $1$, $(\cdot,\cdot)$ represents the $L^2$ inner product in space, the $L^2$ inner product for vector functions is simply defined by ${\mathit{\mathbf{x}}}^T \mathit{\mathbf{y}}=\sum_{i=1}^n \left(x_i ,y_i \right)$ and $Q$, $H_2$, $H_1$, $H_0$ have been given in \eqref{eq3.14}-\eqref{eq3.16}.

	According to \eqref{eq3.18}, the selection of $\alpha$, $\beta$ ensure the positive-definiteness of $\bar{H}_2(\alpha)$, $\bar{H}_1(\beta)$ respectively, and $\bar{H}_0$ has already been positive-definite in Theorem \ref{the3.3}. Then IMEX-RK method \eqref{eq3.10} satisfies the energy dissipation law.

	Now we show that $M_0$ is the uniform upper bound of all $\left\|\tilde{u}^n_i \right\|_{L^{\infty}} $ by proving $\mathcal{E}(\tilde{u}^n_i)\leq \mathcal{E}(\tilde{u}^n_0)$, $i=1,2,...,s-1$ and using lemma \ref{lemma3.1}. Let's perform the same process on $\tilde{u}^n_i$, $i=1,2,...,s-1$ as above, considering the first $m$ lines of $\eqref{eq3.20.1}$, we have
	\begin{align}\label{eq3.23}
		\left(\begin{array}{c}
			\tilde{u}_1^n \\
			\tilde{u}_2^n \\
			\vdots \\
			\tilde{u}_m^n 
		\end{array}\right)=\left(\begin{array}{c}
			\tilde{u}_0^n \\
			\tilde{u}_0^n \\
			\vdots \\
			\tilde{u}_0^n 
		\end{array}\right)+\tau\left(A[m]\left(\begin{array}{c}
			{\mathcal{L}\tilde{u}}_1^n \\
			{\mathcal{L}\tilde{u}}_2^n \\
			\vdots \\
			{\mathcal{L}\tilde{u}}_m^n 
		\end{array}\right)+\hat{A}[m] \left(\begin{array}{c}
			\tilde{N}\left(\tilde{u}_0^n \right)\\
			\tilde{N}\left(\tilde{u}_1^n \right)\\
			\vdots \\
			\tilde{N}\left(\tilde{u}_{m-1}^n \right)
		\end{array}\right)\right),
	\end{align}
	and
	\begin{align}\label{eq3.24}
		\mathcal{E}\left(\tilde{u}^n_m \right)-\mathcal{E}\left(\tilde{u}^n_0 \right)&\leq {-\tilde{\omega} }^T_m H_2[m] \left(P_a \tilde{\omega}_m \right)-\tilde{\omega}^T_m H_1[m] \tilde{\omega}_m +\frac{1}{\tau}\tilde{\omega}^T_m H_0[m] \left(\Delta^{-1} \tilde{\omega}_m \right)\\
		&={-\tilde{\omega} }^T_m \bar{H_2}[m] \left(P_a \tilde{\omega}_m \right)-\tilde{\omega}^T_m \bar{H_1}[m] \tilde{\omega}_m +\frac{1}{\tau}\tilde{\omega}^T_m \bar{H_0}[m] \left(\Delta^{-1} \tilde{\omega}_m \right),\notag
	\end{align}
	where $\tilde{\omega}_m$ denotes a vector which is composed of the first $m$ elements of $\tilde{\omega}$, and
	\begin{align*}
		&H_0[m]={\hat{A} }^{-1}[m] E_L[m],\\
		&H_1[m](\beta)=\beta Q[m]-\frac{L}{2}I,\\
		&H_2[m](\alpha)=\alpha Q[m]-\frac{1}{2}E[m]+\hat{A}^{-1}[m]A[m]E_L[m],\\
		&Q[m]=\left({\hat{A} }^{-1}[m] A[m]-I\right)E_L[m] +I,
	\end{align*}
	because $A$, $\hat{A}$ and $E_L$ are all lower triangle matrices. This make it possible for us to use lemma \ref{lemma3.2}.

	We now illustrate that if condition \eqref{eq3.18} is satisfied, then the matrices $\bar{Q}[m]$, $\bar{H_0}[m]$, $\bar{H_1}[m]$ and $\bar{H_2}[m]$ are all positive-definite, i.e. $\mathcal{E}\left(\tilde{u}^n_m \right)-\mathcal{E}\left(\tilde{u}^n_0 \right)\leq 0$. First, the principal submatrix of a real symmetric positive-definite matrix is still positive-definite, so $\bar{H_0}[m]$ and $\bar{Q}[m]$ are positive-definite. Then, according to Lemma \ref{lemma3.2}, we have $\lambda_{min}(\bar{Q})\leq \lambda_{min}(\bar{Q}[m])$, implying that 
	the choice of the stabilizer parameter $\alpha$ and $\beta$ can ensure the positive definiteness of $\bar{H_2}[m]$ and $\bar{H_1}[m]$,  thereby obviating the need to recompute the stabilizer values for intermediate numerical solutions.

	With the help of lemma \ref{lemma3.1}, we have $\left\|\tilde{u}^n_i \right\|_{L^{\infty}}\leq M(\tilde{\phi}^0,\Omega)=M_0$, $i=1,2,\dots,s$. Performing similar process for $n=0,1,\dots,\frac{T}{\tau}$, the uniform boundedness at any time step is obtained. So $f(u^n_i)=\tilde{f}(\tilde{u}^n_i)$ for $i=0,1,...,s,\; n=0,1,...,\frac{T}{\tau}$, then IMEX-RK method \eqref{eq2.8} and \eqref{eq3.10} share the same numerical solution. 
\end{proof}
\medskip

\begin{remark}
	Actually, introducing an auxiliary problem by truncating $F$ into quadratic growth at infinities can also be viewed as a mathematical induction. By the Taylor expansion, we have 
	\begin{align*}
		F\left( u\right)-F\left( v\right)=\left(v^3 -\left(\varepsilon +a\right)v \right)\left(u-v \right)+\frac{1}{2}\left(3w^2-\left(\varepsilon +a\right)\right)\left(u-v \right) ^2,
	\end{align*}
	where $w$ is between $u$ and $v$ pointwisely. Because of the explicit treatment of $f$ in \eqref{eq2.8}, $w$ always lies in the known numerical solutions, so its uniform boundedness is trivial due to inductive hypotheses.
\end{remark}
\begin{remark}
	For any matrix $M \in \mathbb{C}^{n\times n}$, $M$ is positive-definite if and only if $\frac{1}{2}(M+M^*)$ is positive-definite, here $M^*$ represents the conjugate transpose of $M$. The reason why we base on $\bar{M}$ but not $M$ in Theorem \ref{the3.3} is that, we can facilitate the discussion of eigenvalue of $H_1(\beta)$ and $H_2(\alpha)$ from \eqref{M} and \eqref{Mbar} without changing the result of inequality \eqref{eq3.22} and \eqref{eq3.24}. This also make it feasible to use Cauchy interlace theorem.
\end{remark}
\begin{remark}
	Actually without imposing the Lipschitz condition, it is still possible to establish the uniform boundedness of the solution by leveraging error estimates combined with the mathematical induction. However, this approach does not facilitate the derivation of unconditional energy dissipation. In this work,  we can obtain unconditional energy dissipation without imposing Lipschitz condition in Theorem \ref{the3.3}, as the uniform boundedness is recovered through the energy dissipation analysis rather than reliance on error estimates.
\end{remark}

\section{Error analysis for PFC model}
In this section, we will prove the $L^{\infty}$ convergence for the constructed IMEX-RK scheme \eqref{eq2.8}. 

\medskip
\begin{theorem}\label{the4.1}
	Suppose that \eqref{eq3.14}-\eqref{eq3.18} are satisfied. Then for a $s$-stage IMEX-RK scheme \eqref{eq2.8} which is of order $p$, suppose that $\phi_e$ is smooth enough and $\tau$ satisfies an $O(1)$ constraint, we have the following error estimate:
	\begin{align}\label{eq4.1}
		\left\|\phi_e\left(t_n\right)-\phi^n\right\|_{L^{\infty}}\leq Ce^{CT}\tau^p,
	\end{align}
	where $C$ is a positive constant dependent on the stabilizers and the smoothness of $\phi_e$ but independent of $\tau$. 
\end{theorem}

\medskip
\begin{proof}
	The system of the IMEX-RK method \eqref{eq2.8} has the following matrix form:
	\begin{align}
		u_0^n & =\phi^n,\label{eq4.2}\\
		\left(\begin{array}{c}
			u_1^n \\
			u_2^n \\
			\vdots \\
			u_s^n 
		\end{array}\right)&=\left(\begin{array}{c}
			u_0^n \\
			u_0^n \\
			\vdots \\
			u_0^n 
		\end{array}\right)+\tau\left(A\left(\begin{array}{c}
			{Lu}_1^n \\
			{Lu}_2^n \\
			\vdots \\
			{Lu}_s^n 
		\end{array}\right)+\hat{A} \left(\begin{array}{c}
			N\left(u_0^n \right)\\
			N\left(u_1^n \right)\\
			\vdots \\
			N\left(u_{s-1}^n \right)
		\end{array}\right)\right),\label{eq4.3}
	\end{align}	
	In order to obtain local truncation error, we define intermediate functions $\bar{u}_i^n$ by replacing $\phi^n$ with exact solution at corresponding time in \eqref{eq2.8}:
	\begin{align}
		\bar{u}_0^n & =\phi_e(t_n),\label{eq4.4}\\
		\left(\begin{array}{c}
			\bar{u}_1^n \\
			\bar{u}_2^n \\
			\vdots \\
			\bar{u}_s^n 
		\end{array}\right)&=\left(\begin{array}{c}
			\bar{u}_0^n \\
			\bar{u}_0^n \\
			\vdots \\
			\bar{u}_0^n 
		\end{array}\right)+\tau\left(A\left(\begin{array}{c}
			{\mathcal{L}\bar{u}}_1^n \\
			{\mathcal{L}\bar{u}}_2^n \\
			\vdots \\
			{\mathcal{L}\bar{u}}_s^n 
		\end{array}\right)+\hat{A} \left(\begin{array}{c}
			N\left(\bar{u}_0^n \right)\\
			N\left(\bar{u}_1^n \right)\\
			\vdots \\
			N\left(\bar{u}_{s-1}^n \right)
		\end{array}\right)\right).\label{eq4.5}
	\end{align}
	We define that $e_{i}^n=\bar{u}_{i}^n-u_{i}^n$, $i=0,1,...,s-1$, $e_{s}^n=\phi_e(t_{n+1})-u_{s}^n=r_{n+1}+\bar{u}_{s}^n-u^n_s$, where $r_{n+1}=\phi_e(t_{n+1})-\bar{u}^n_s$ is the local truncation error of order $p+1$.

	By computing the difference between \eqref{eq4.3} and \eqref{eq4.5}, we have
	\begin{align}\label{eq4.6}
		\left(\begin{array}{c}
			e_1^n \\
			e_2^n \\
			\vdots \\
			e_s^n 
		\end{array}\right)&=\left(\begin{array}{c}
			e_0^n \\
			e_0^n \\
			\vdots \\
			e_0^n 
		\end{array}\right)+\tau \left(A\left(\begin{array}{c}
			{\mathcal{L}e}_1^n \\
			{\mathcal{L}e}_2^n \\
			\vdots \\
			{\mathcal{L}e}_s^n 
		\end{array}\right)+\hat{A}\left(\begin{array}{c}
			N\left(\bar{u}_0^n\right)-N\left(u_0^n \right)\\
			N\left(\bar{u}_1^n\right)-N\left(u_1^n \right)\\
			\vdots \\
			N\left(\bar{u}_{s-1}^n\right)-N\left(u_{s-1}^n \right)
		\end{array}\right)\right)
		+\left(\begin{array}{c}
			0\\
			0\\
			\vdots\\
			r_{n+1} 
		\end{array}\right).
	\end{align}
	Following the similar process in \cite[Theorem 4.1]{fu2024energydiminishingimplicitexplicitrungekutta} and using the uniform boundedness of numerical solutions in Theorem \ref{the3.3}, we can obtain
	\begin{align}\label{eq4.7}
		&\quad\ \frac{1}{2}\left(e_s^n ,P_a e_s^n \right)-\frac{1}{2}\left(e_s^0 ,P_a e_s^0 \right)\\
		&=\frac{1}{\tau }q^T H_0 \left(\Delta^{-1} q\right)-\beta q^T Qq-q^T H_2 \left(P_a q\right)-q^T \left(\begin{array}{c}
			f\left({\bar{u} }_0^n \right)-f\left(u_0^n \right)\\
			f\left({\bar{u} }_1^n \right)-f\left(u_1^n \right)\\
			\vdots \\
			f\left({\bar{u} }_{s-1}^n \right)-f\left(u_s^n \right)
		\end{array}\right)-\frac{1}{\tau }q^T{\hat{A} }^{-1} \left(\begin{array}{c}
			0\\
			0\\
			\vdots \\
			r_{n+1}
		\end{array}\right)\notag
	\end{align}
	where $q=(e_1-e_0,e_2-e_1,...,e_s-e_{s-1})^T$ and $\Delta^{-1}$ is omitted form $r_{n+1}$ because the order of $r_{n+1}$ is still.

	Now we show that $\left\|e_j^n \right\|_{L^{\infty}} $ and $\left\|\nabla e^n_j \right\|_{L^2}$ can be governed by $\left(e_j^n ,P_a e_j^n \right)$, where stabilizer $a$ plays an important role. Observing that every $e_j^n$ has zero mean, using \eqref{eq3.8} and \eqref{eq3.9} in lemma \ref{lemma3.1}, we have
	\begin{gather}
		\left\|e_j^n \right\|_{L^2} \leq \frac{C_p}{\epsilon} \left\|\Delta e_j^n\right\|_{L^2 },\label{eq4.8}\\
		{\left\|\nabla e_j^n\right\|}_{L^2 } \leq \left( \frac{C_p}{2}+\frac{1}{2\epsilon}\right) {\left\|\Delta e_j^n\right\|}_{L^2 }.\label{eq4.9}
	\end{gather}
	So we only need to check whether $\left\|\Delta e_j^n \right\|_{L^2} $ can be controlled by $\left(e_j^n ,P_a e_j^n \right)$ by using \eqref{eq3.2}. Utilizing inequality \eqref{eq3.1}, we have
	\begin{align}\label{eq4.10}
		\left(e_j^n,P_a e_j^n \right)&=\left(e_j^n,\Delta^2 e_j^n \right)+2\left(e_j^n, \Delta e_j^n\right)+\left(e_j^n,e_j^n \right)+a\left(e_j^n,e_j^n \right)   \\
		&\geq\left(e_j^n,\Delta^2 e_j^n \right)-2p\left(e_j^n,\Delta^2 e_j^n \right)-2q\left(e_j^n,e_j^n \right)+\left(1+a \right) \left(e_j^n,e_j^n \right)\notag\\
		&=\left(1-2p \right)\left\|\Delta e_j^n \right\|^2_{L^2}\notag.
	\end{align}
	For any $a>0$, we may select $q=\frac{1}{2}+\frac{1}{2}a$, what makes $1-2p>0$.

	For the last two terms in \eqref{eq4.7}, we have
	\begin{align}
		-\frac{1}{\tau }q^T{\hat{A} }^{-1} \left(\begin{array}{c}
			0\\
			0\\
			\vdots \\
			r_n 
		\end{array}\right)&\leq \frac{C_1}{\tau}\left\| \nabla^{-1}q \right\|_{L^2}^2+\frac{C_2}{\tau} \left\|\nabla r_{n+1}\right\|_{L^2}^2,\label{eq4.11}\\
		-q^T \left(\begin{array}{c}
			f\left({\bar{u} }_0^n \right)-f\left(u_0^n \right)\\
			f\left({\bar{u} }_1^n \right)-f\left(u_1^n \right)\\
			\vdots \\
			f\left({\bar{u} }_s^n \right)-f\left(u_s^n \right)
		\end{array}\right)
		&\leq \frac{C_3}{\tau}\left\|\nabla^{-1}q \right\|_{L^2}^2 +C_4\tau \left\|\nabla q \right\|_{L^2}^2+C_5\left\|\nabla e^n_0 \right\|_{L^2}^2\label{eq4.12}\\
		&\leq \frac{C_3}{\tau}\left\|\nabla^{-1}q \right\|_{L^2}^2 +C_5\tau (q,P_a q)+C_7(e^n_0,P_a e^n_0).\notag
	\end{align}
	Combining inequalities \eqref{eq4.7}, \eqref{eq4.11} and \eqref{eq4.12}, setting $C_1+C_3\leq \lambda_{min}(H_0)$ and $C_5\tau \leq \lambda_{min}(H_2)$, we have
	\begin{align}\label{eq4.13}
		\left(e_s^n ,P_a e_s^n \right)\le \left(1+C\tau \right)\left(e_0^n ,P_a e_0^n \right)+\frac{C^{\prime } }{\tau }{\left\|{\nabla r}_{n} \right\|}_{L^2 }^2,
	\end{align}
	Finally, leveraging Gronwall's inequality, we derive the error estimate \eqref{eq4.1} under an $O(1)$ bound on the time step size $\tau$, with the Lipschitz condition being recovered by uniform boundedness.
\end{proof}

\begin{remark}
	\romannumeral1) For PFC model, we derive the energy dissipation together with a priori uniformly bound estimate, so the $L^{\infty}$-convergence can be easily performed. This result is significant because the energy dissipation and uniform boundedness is independent of $\tau$ and $T$, and the Lipschitz continuity condition is no longer required.

	\romannumeral2) If obtaining an a priori uniform bound for the gradient flow appears difficult, but the $L^{\infty}$-convergence can be established, then mathematical induction can be employed to achieve uniform boundedness of the numerical solutions, thereby removing the need for Lipschitz continuity. However, the proof becomes constrained by conditions on the time step size and the final time due to convergence requirements, causing the energy dissipation to no longer be unconditional or global-in-time, despite the absence of a Lipschitz condition.

	\romannumeral3) For the gradient flows with $L^{\infty}$-bounds on the numerical solution, if we can truncate nonlinear term like \eqref{eq3.12} and \eqref{eq3.13}, an unconditional energy dissipation can be presented without Lipschitz continuity. Hence for the error estimate, mathematical induction is unnecessary for error analysis, so we can weaken the constraint on time step.
\end{remark}

\begin{remark}
	Our analysis is based on the high-order temporal discretization in this work, the framework of employing Lemmas \ref{lemma3.1} and \ref{lemma3.2} to achieve unconditional uniform boundedness,  followed by deriving $L^{\infty}$-error estimates using energy methods, is broadly feasible to match with any spatial discretization methods, such as finite element method,  finite difference method and spectral method. 
\end{remark}

\section{Numerical experiments}
This section devotes to some numerical experiments conducted by IMEX-RK method \eqref{eq2.8} for PFC model. The periodic boundary conditions has been used on all the computational domains so we can employ Fourier spectral method as spatial discretization with fast Fourier transform to enhance computation. The stabilizers are set to $\alpha=0$ and $\beta=1$ unless otherwise stated. Here we use the energy-decreasing four-stage third-order IMEX-RK method \cite[Section 5]{fu2024energydiminishingimplicitexplicitrungekutta}.
\subsection{Convergence test}
We consider the PFC model \eqref{eq2.6} and demonstrate the convergence of the proposed method on $\Omega=(0,32)\times(0,32)$, with $\varepsilon = 0.025$, $a=1,0.5,0.1,0.001$ and the following smooth initial data:
\begin{align*}
	\phi_0=0.05-0.01\mathrm{cos}\left( \frac{2\pi x}{32}\right) \mathrm{cos}\left( \frac{2\pi y}{32}\right) .
\end{align*}

Set the $256\times256$ spatial mesh to ensure the spatial error is small enough so that we may ignore it compared to the temporal errors, we compute the numerical solution at $T=2$ with $\tau=2^{-k}, k=3,4,...,8$. Since we have no exact solution, Cauchy error is chosen to verify the convergence rate. The $l^{\infty}$ error of numerical solution will be released in Table \ref{table5.1}. As observed, the temporal convergence is evident and the decreasing of $a$ increases the convergence accuracy without influencing the $l^{\infty}$-convergence in Theorem \ref{the4.1}.
\begin{table}[htb]
	\caption{$l^{\infty}$-errors and convergence rates for four-stage third-order IMEX-RK method \eqref{eq2.8} in temporal direction with $\epsilon=0.025$, $\tau=2^{-k}$, $k=3,4,...,8$ and $a=1,0.5,0.1,0.001$ at $T=2$}
	\centering
	\begin{tabular}{ccccccccc}
		\toprule
		$\tau$ & $a=1$& Rate & $a=0.5$& Rate & $a=0.1$& Rate & $a=0.001$ & Rate \\
		\midrule
		$2^{-4}$&7.80542E-08&---&3.03738E-08&---&9.76150E-09&---&6.70411E-09&--- \\
		$2^{-5}$&1.20764E-08&2.6923&4.47767E-09&2.7620&1.37721E-09&2.8254&9.35122E-10&2.8418\\
		$2^{-6}$&1.70851E-09&2.8214&6.13750E-10&2.8670&1.83772E-10&2.9058&1.23953E-10&2.9154 \\
		$2^{-7}$&2.28613E-10&2.9018&8.05862E-11&2.9290&2.37671E-11&2.9505&1.59727E-11&2.9561 \\
		$2^{-8}$&2.96266E-11&2.9479&1.03329E-11&2.9633&3.02248E-12&2.9752&2.02777E-12&2.9776 \\
		\bottomrule
	\end{tabular}
	\label{table5.1}
\end{table}

\subsection{Energy stability test}
Now we will show how the original energy \eqref{eq2.3} of PFC model \eqref{eq2.6} evolves under different stabilizers and time steps. We consider the following smooth initial condition:
\begin{align*}
	\phi_0&=0.07-0.02\,\mathrm{cos}\left(\frac{\pi(x-12)}{16} \right)\mathrm{sin}\left(\frac{\pi(y-1)}{16}\right)\\
	&\quad+ 0.02\left(\mathrm{cos}\left( \frac{\pi(x+10)}{32}\right)\mathrm{cos}\left(\frac{\pi(y+3)}{32}\right) \right) ^2-0.01\left(\mathrm{sin}\left( \frac{\pi x}{8}\right)\mathrm{sin}\left(\frac{\pi(y-6)}{8}\right) \right) ^2.
\end{align*}
on $\Omega=(0,128)\times(0,128)$. We test energy evolutions with  different stabilizers and  time steps under the fix parameters 
$\epsilon=0.025$, $T=120$ and  $256\times256$  spatial mesh. 
\begin{figure}[htb]
	\centering
	{
		\begin{minipage}[b]{.48\linewidth}
			\centering
			\includegraphics[scale=0.45]{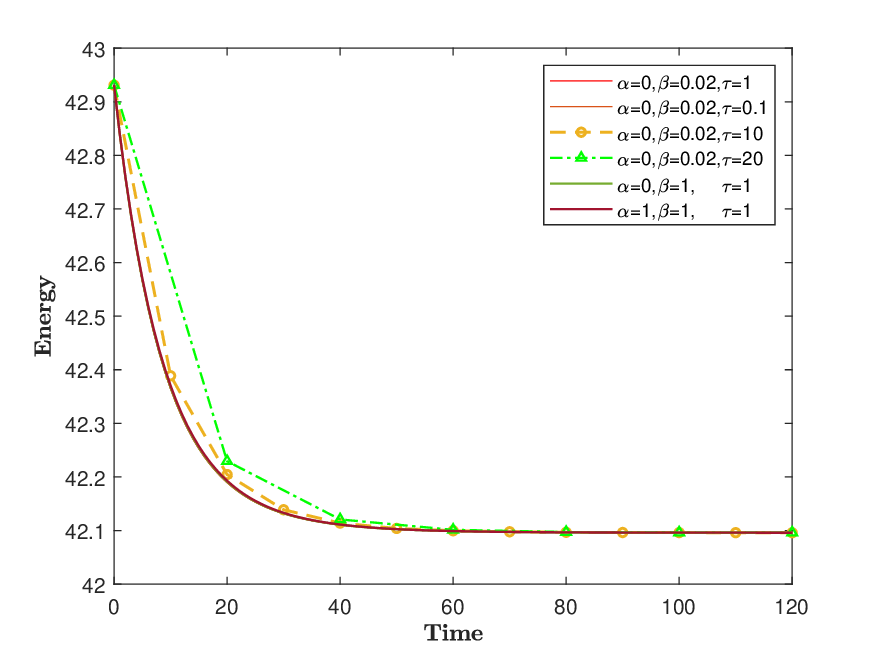}
		\end{minipage}
	}
	{
		\begin{minipage}[b]{.48\linewidth}
			\centering
			\includegraphics[scale=0.45]{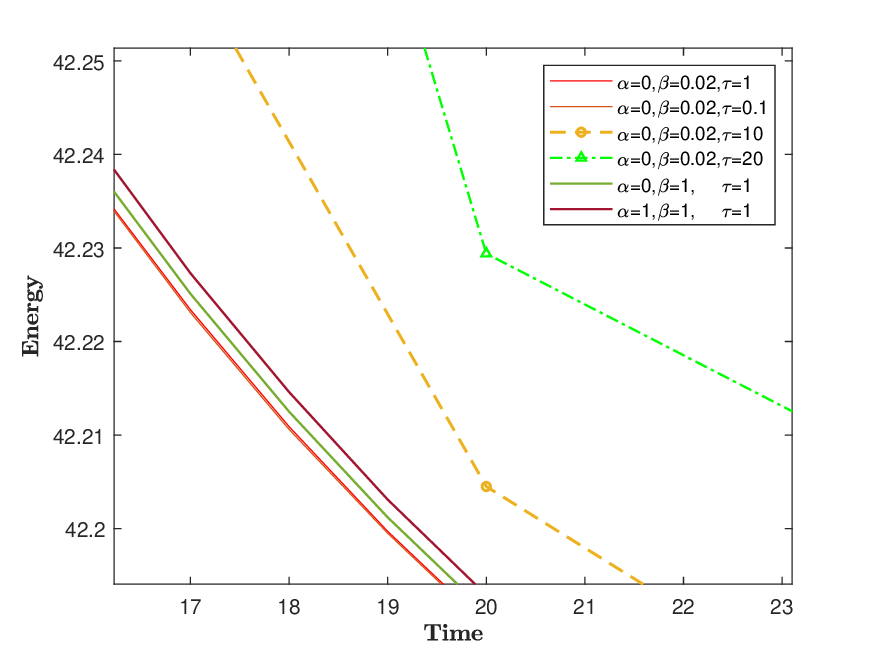}
		\end{minipage}
	}
	\caption{Evolution results of the original energy for the four-stage third-order IMEX-RK method \eqref{eq2.8} with $\epsilon=0.025$, $a=0.001$, $\tau=0.5$, different stabilizers and time steps at $T=120$}
	\label{figure5.1}
\end{figure}

In Figure \ref{figure5.1}, we present the evolution of the original energy under different parameters. Even though the energy is decreasing under different parameter selection, the larger stabilizers and longer time step will cause a slower energy dissipative progress. Actually, for the simple initial condition mentioned above, inappropriate stabilizers and larger time step will not affect the final result. In more complex situations, stabilizers that do not meet the conditions in Theorem \ref{the3.3} may cause energy to increase at certain times and oversized time step may fail to capture the evolution process of numerical solutions correctly.
\subsection{2D phase transition behaviors}
In this experiment, we simulate the phase transition behaviors on domain $\Omega=(0,128)\times(0,128)$. We compute the numerical solution at $T=2000$ with $\epsilon=0.025$, $a=0.001$ and $\tau=0.1$, and use $256\times256$ spatial mesh. The initial data is set to a random perturbation \cite{Zhang2024pfc}:
\begin{align*}
	\phi_0=0.06+0.01\,\mathrm{rand}(x,y),
\end{align*}
where $\mathrm{rand}(x,y)$ is a randomly chosen number between $-1$ and $1$. The phase transition behaviors evolution and corresponding energy have been shown in Figure \ref{figure5.2} and Figure \ref{figure5.3}. 
\begin{figure}[htb]
	\centering
	\subfigure[$t=100$]
	{
		\begin{minipage}[b]{.27\linewidth}
			\centering
			\includegraphics[scale=0.3]{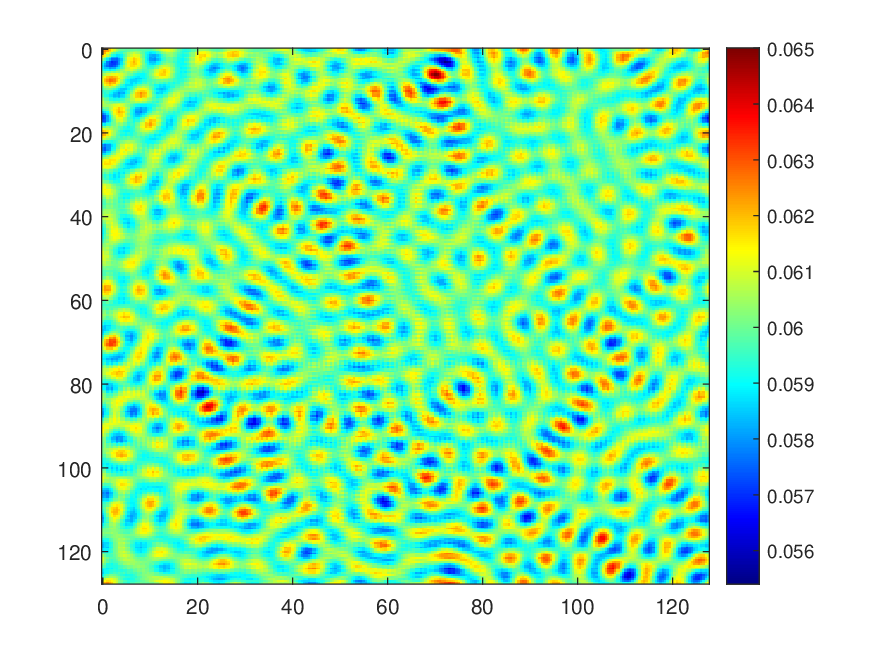}
		\end{minipage}
	}
	\subfigure[$t=400$]
	{
		\begin{minipage}[b]{.27\linewidth}
			\centering
			\includegraphics[scale=0.3]{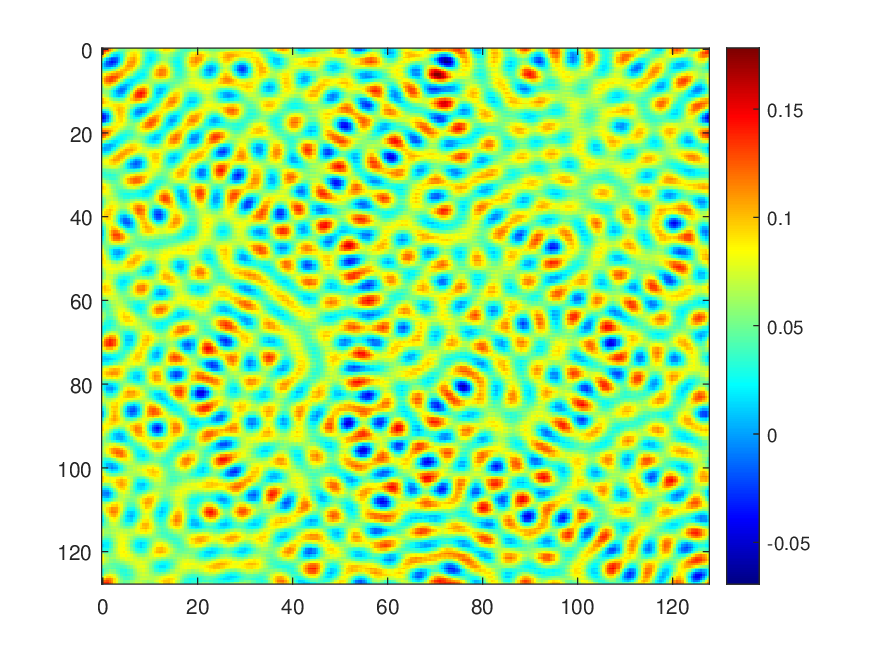}
		\end{minipage}
	}
	\subfigure[$t=600$]
	{
		\begin{minipage}[b]{.27\linewidth}
			\centering
			\includegraphics[scale=0.3]{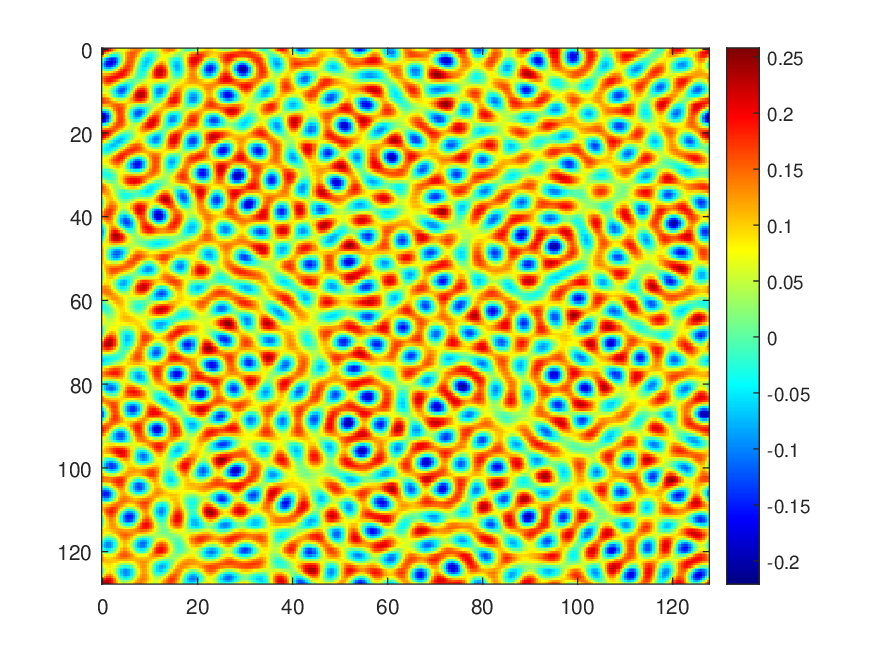}
		\end{minipage}
	}
	\subfigure[$t=800$]
	{
		\begin{minipage}[b]{.27\linewidth}
			\centering
			\includegraphics[scale=0.3]{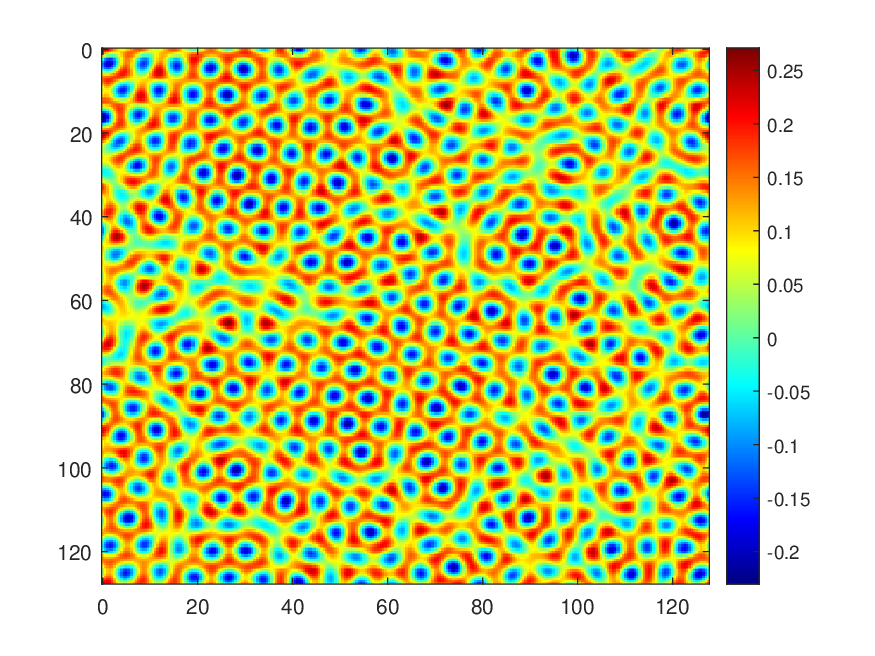}
		\end{minipage}
	}
	\subfigure[$t=1200$]
	{
		\begin{minipage}[b]{.27\linewidth}
			\centering
			\includegraphics[scale=0.3]{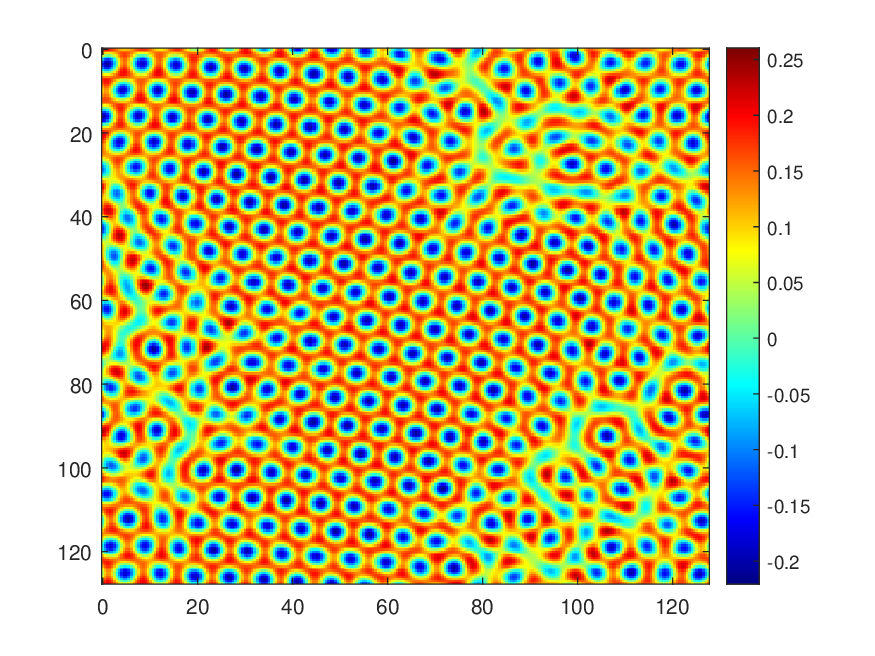}
		\end{minipage}
	}
	\subfigure[$t=2000$]
	{
		\begin{minipage}[b]{.27\linewidth}
			\centering
			\includegraphics[scale=0.3]{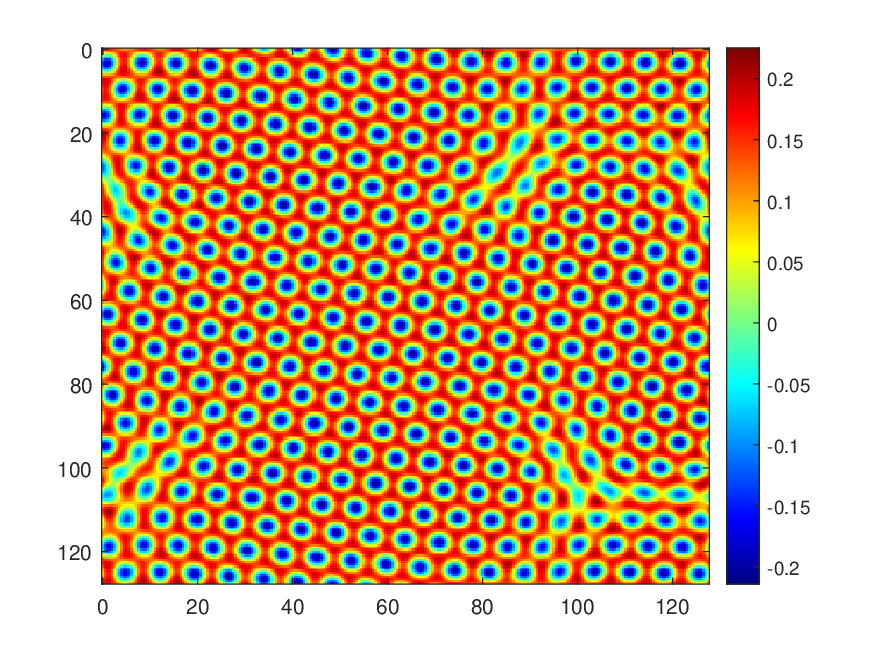}
		\end{minipage}
	}
	\caption{The snapshots of $\phi(x,y,t)$ using the four-stage third-order IMEX-RK method \eqref{eq2.8} with $\epsilon=0.025$, $a=0.001$ and $\tau=0.1$ at $t=100,400,600,800,1200,2000$, respectively.}
	\label{figure5.2}
\end{figure}
\begin{figure}
	\centering
	{
		\begin{minipage}[b]{.48\linewidth}
			\centering
			\includegraphics[scale=0.45]{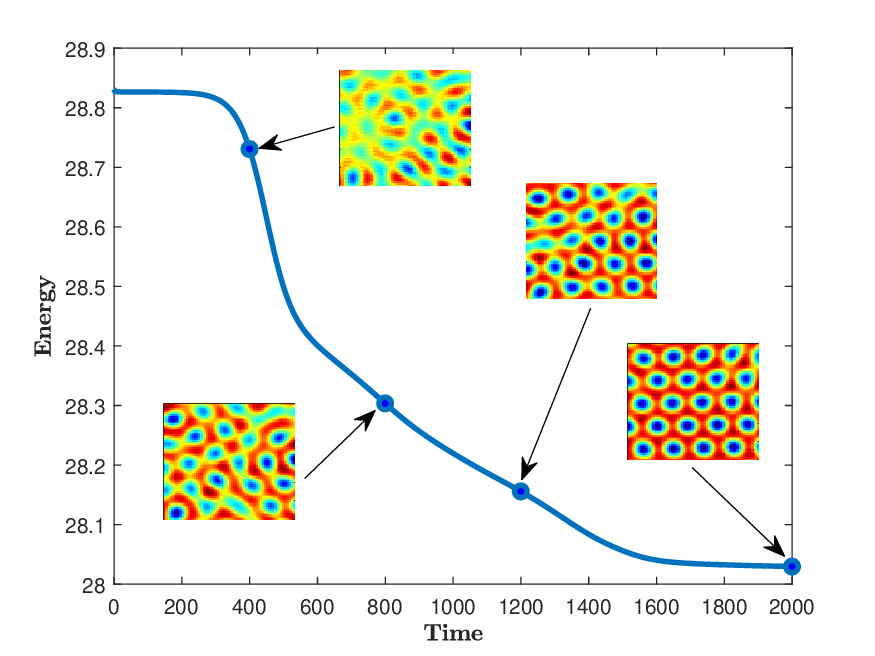}
		\end{minipage}
	}
	{
		\begin{minipage}[b]{.48\linewidth}
			\centering
			\includegraphics[scale=0.45]{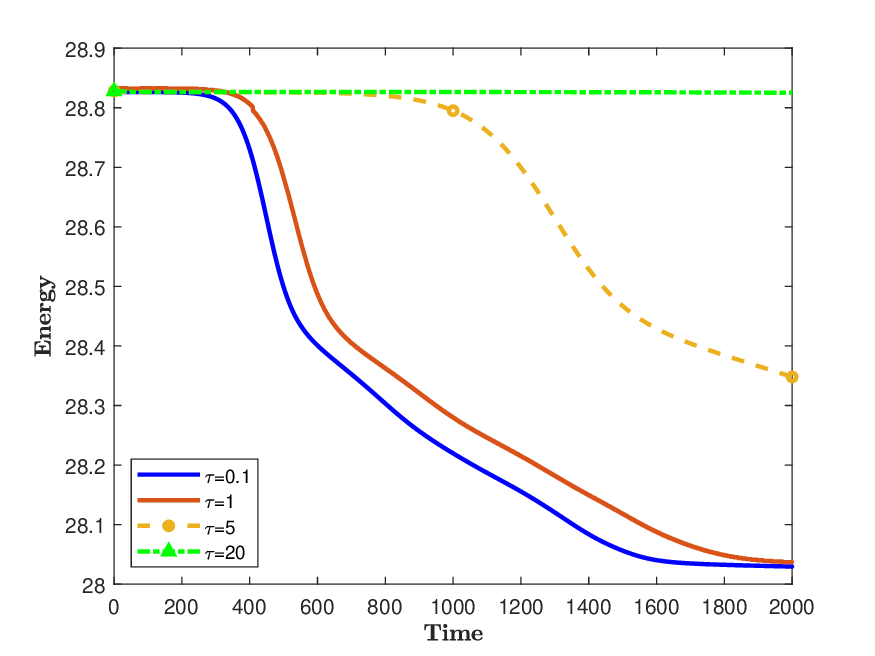}
		\end{minipage}
	}
	\caption{Energy evolution result for the phase behaviors. The inserting figure (left) showing the evolution of phase transition at $t=400,800,1200,2000$ and evolution of the energies (right) with $\tau=0.1,1,5,20$.}
	\label{figure5.3}
\end{figure}

At $t=400$, the hexagonal phase initially emerges and expands over time. Eventually, the entire computational domain almost becomes full filled with hexagonal phase. As can be seen from the right figure of Figure \ref{figure5.3}, oversized $\tau$ leads to a severe delay of energy dissipation. But for each $\tau$, corresponding to  the unconditional energy dissipation in Theorem \ref{the3.3}, the energy is still decreasing, including the case of $\tau=20$ which is hard to distinguish due to the scale of chart.
\subsection{2D crystal growth simulation}
In this simulation, we simulate the evolution of three crystallites
with different orientations, which leads to a complex dynamical process involving the motion of liquid-crystal interfaces and grain boundaries separating the crystals.

We consider the following initial crystallites \cite{Zhang2024pfc,Qiao2024pfc}
\begin{align*}
	\phi_0\left( x_l,y_l\right) =\bar{\phi}+C\left( \mathrm{cos}\left( \frac{p}{\sqrt{3}}y_l\right)\mathrm{cos}\left( px_l\right)-0.5\,\mathrm{cos}\left(\frac{2p}{\sqrt{3}}y_l \right) \right),
\end{align*}
where $x_l$ and $y_l$ define a local system of coordinates oriented with the crystallite lattice. To generate crystallites with different orientations, we define the local coordinates $(x_l,y_l)$ using an affine transformation of the global coordinates $(x,y)$ with a rotation angle $\theta$, i.e.,
\begin{align*}
	&x_l=x\,\mathrm{sin}\left(\theta \right)+y\,\mathrm{cos}\left( \theta\right), \\
	&y_l=-x\,\mathrm{cos}\left(\theta \right) +y\,\mathrm{sin}\left( \theta\right) ,
\end{align*}
where $\theta=-\frac{\pi}{4}, 0, \frac{\pi}{4}$ for $l=1,2,3$, respectively, and $\phi_0=0.285$, $C=0.446$, $p=0.66$, $a=0.001$ and $\epsilon=0.25$.

The initial configuration is designed as follows. First we set $\phi_0=0.285$ all over the domain $\Omega$, Then we modify this constant configuration by setting three perfect crystallites in three small square patches with size of $30\times30$, lying in the computational domain, as illustrated by the first picture in Figure \ref{figure5.6}. The computational domain $\Omega=(0,512)\times(0,512)$, we choose  $512\times512$ spatial mesh and time step $\tau=0.1$.
\begin{figure}[htb]
	\centering
	\subfigure[$t=0$]
	{
		\begin{minipage}[b]{.3\linewidth}
			\centering
			\includegraphics[scale=0.32]{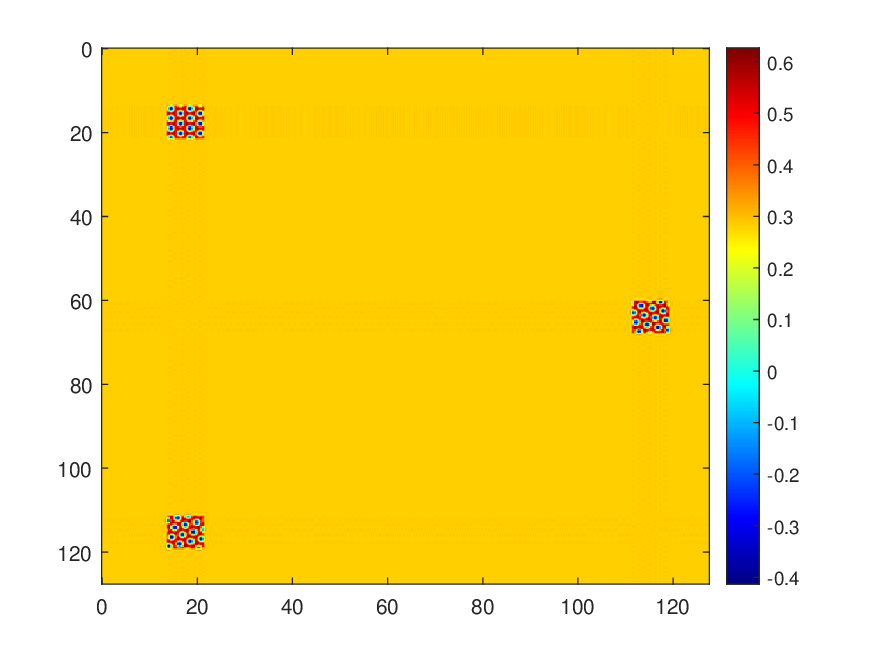}
		\end{minipage}
	}
	\subfigure[$t=100$]
	{
		\begin{minipage}[b]{.3\linewidth}
			\centering
			\includegraphics[scale=0.32]{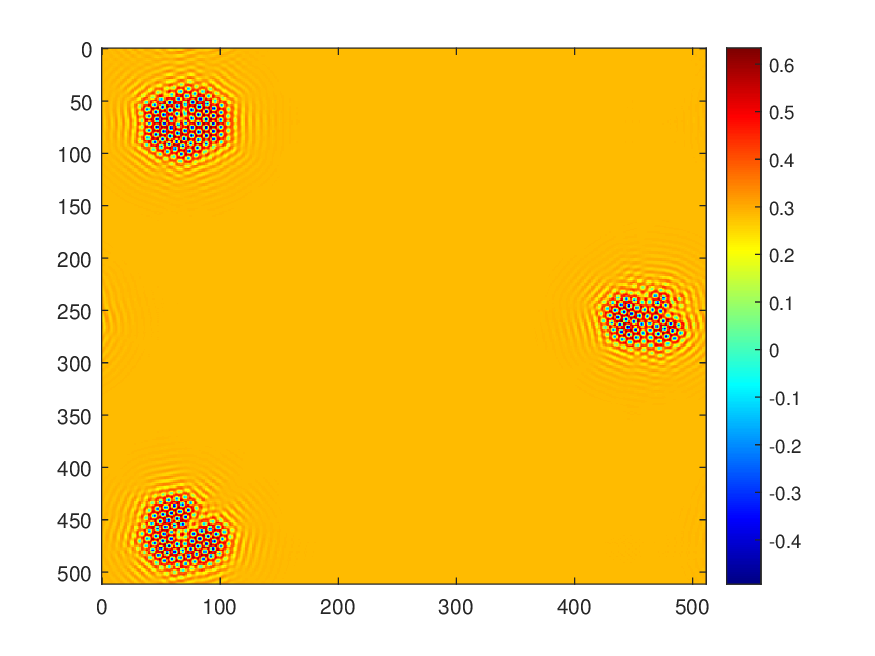}
		\end{minipage}
	}
	\subfigure[$t=200$]
	{
		\begin{minipage}[b]{.3\linewidth}
			\centering
			\includegraphics[scale=0.32]{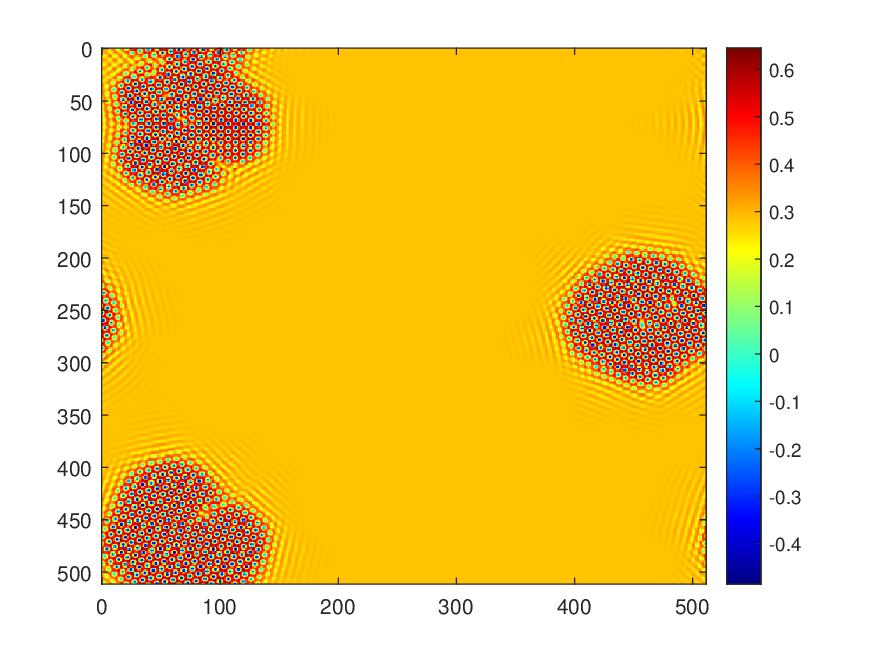}
		\end{minipage}
	}
	\subfigure[$t=400$]
	{
		\begin{minipage}[b]{.3\linewidth}
			\centering
			\includegraphics[scale=0.32]{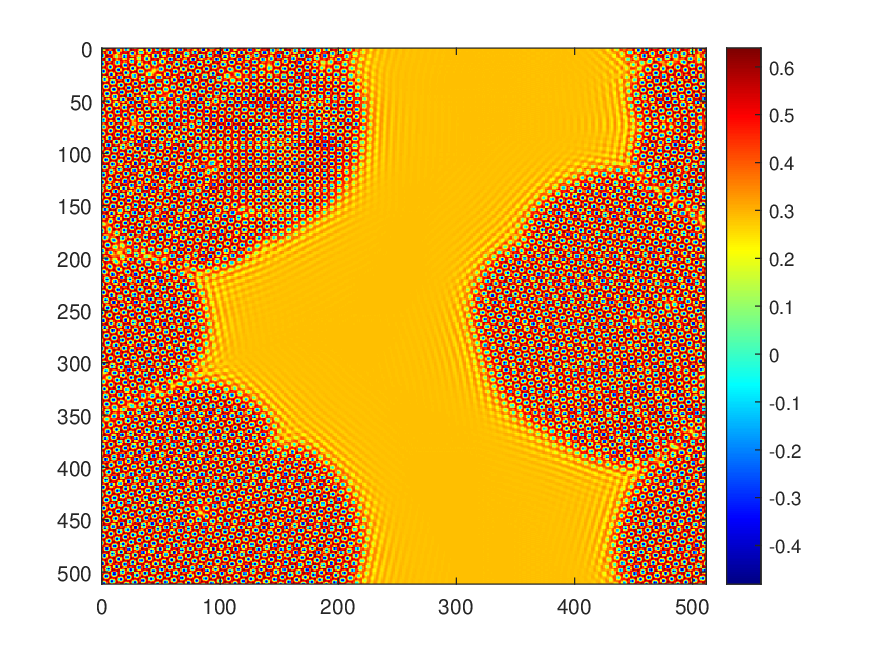}
		\end{minipage}
	}
	\subfigure[$t=600$]
	{
		\begin{minipage}[b]{.3\linewidth}
			\centering
			\includegraphics[scale=0.32]{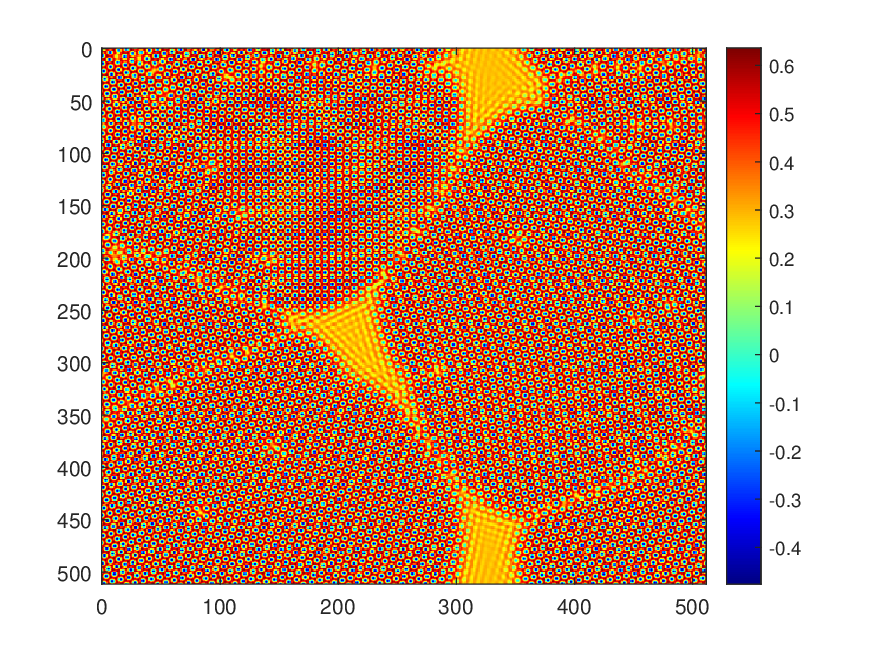}
		\end{minipage}
	}
	\subfigure[$t=1500$]
	{
		\begin{minipage}[b]{.3\linewidth}
			\centering
			\includegraphics[scale=0.32]{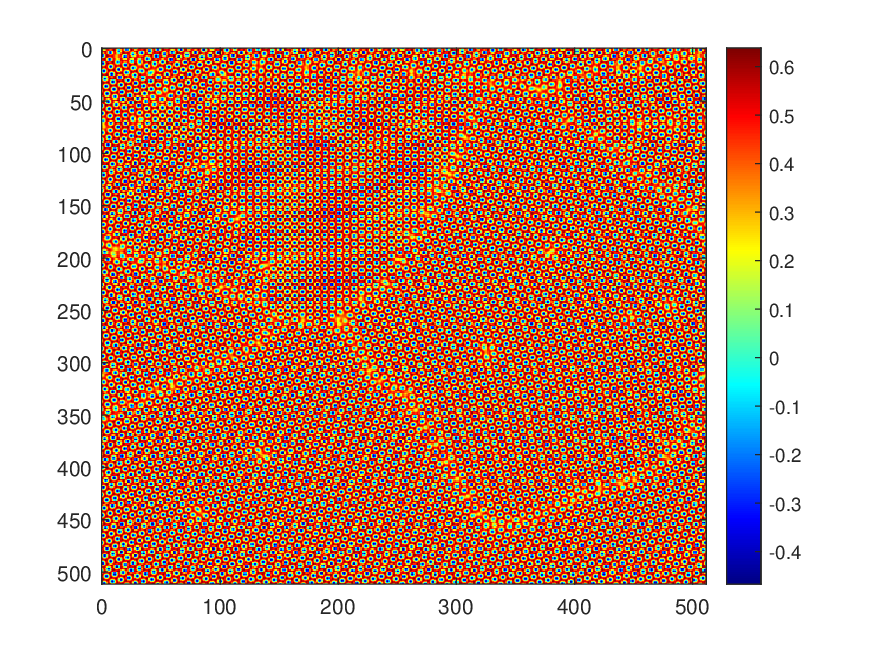}
		\end{minipage}
	}
	\caption{The evolution of crystal growth using the four-stage third-order IMEX-RK method \eqref{eq2.8} with $\epsilon=0.25$, $a=0.001$ at $t=0,100,200,400,600,1500$, respectively.}
	\label{figure5.6}
\end{figure}
\begin{figure}
	\centering
	\includegraphics[width=0.67\textwidth]{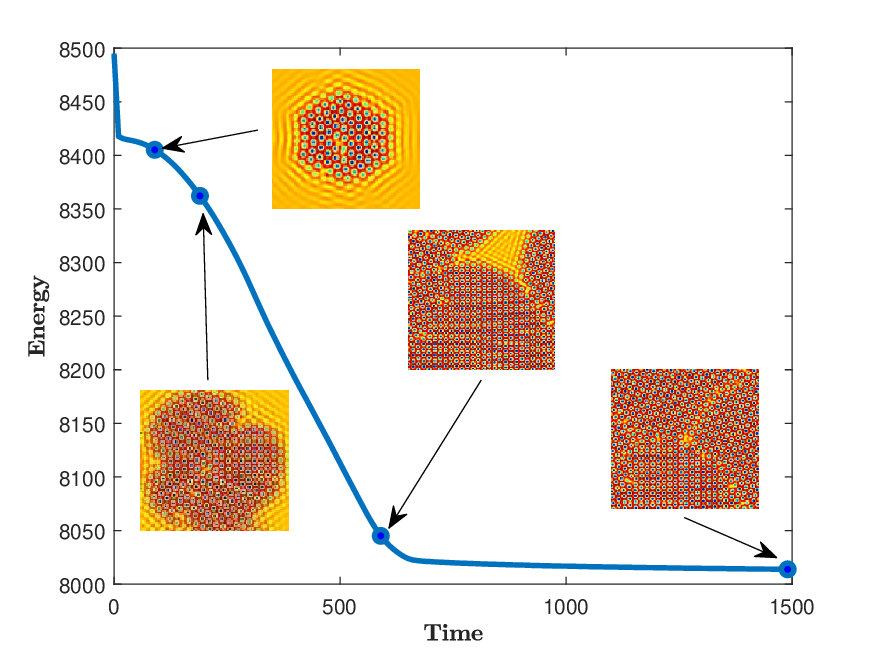}
	\caption{Energy evolution result for the crystal growth. The inserting figure is the evolution of crystal growth at $t=100,200,600,1500$.}
	\label{figure5.7}
\end{figure}

In Figure \ref{figure5.6}, snapshots of the dynamic evolution of crystal growth at time points $t=0,100,200,400,$ $600,1500$ are presented respectively. We can observe that the hexagonal phase is gradually forming between crystallites and finally impinge upon another one at $t=400$, which leads to the grain boundaries caused by the different orientations of the crystallites we choose before. The energy evolution is showed in Figure \ref{figure5.7}, agreeing with the theoretical results. Similar results are also reported in, e.g. \cite{elder2002modeling,elder2004modeling,HU20095323,10.1016/j.jcp.2016.10.020,Qiao2024pfc,Zhang2024pfc}.
\subsection{Effect of \texorpdfstring{$\tau$}{τ} on pattern formation}
Actually, in \eqref{eq1.1}, a third-order term $-\frac{r}{3}\phi^3$ was eliminated in the original PFC model, while retaining $\phi^4$ for simplify, because the third-order term does not affect the qualitative features of the model. However, the third-order term is essential for
accurate simulations of material properties \cite{guo2015modified,chen2018phase}.

Depending on the values of $\epsilon$ and initial conditions according to phase diagrams, referring to \cite{emdadi2016revisiting,guo2015modified,asadi2015review}, PFC model can generate different patterns, such as  striped, hexagonal, and their coexisting phases. In order to examine the evolution of the PFC model with the third-order term from a
random nonequilibrium state to a steady-state pattern structure, we set an initial condition as $0.1+0.5\,\mathrm{rand}(x,y)$ on a small domain, which we take a small hexagon at the center of $\Omega=(0,128)\times(0,128)$ \cite{chen2021phase} in the following test. Then we assign a constant $\psi_0$ to the rest of $\Omega$, and we use $a=0.001$ $\tau=0.1$, $T=1000$.
\begin{figure}[htb]
	\centering
	\includegraphics[width=0.92\textwidth]{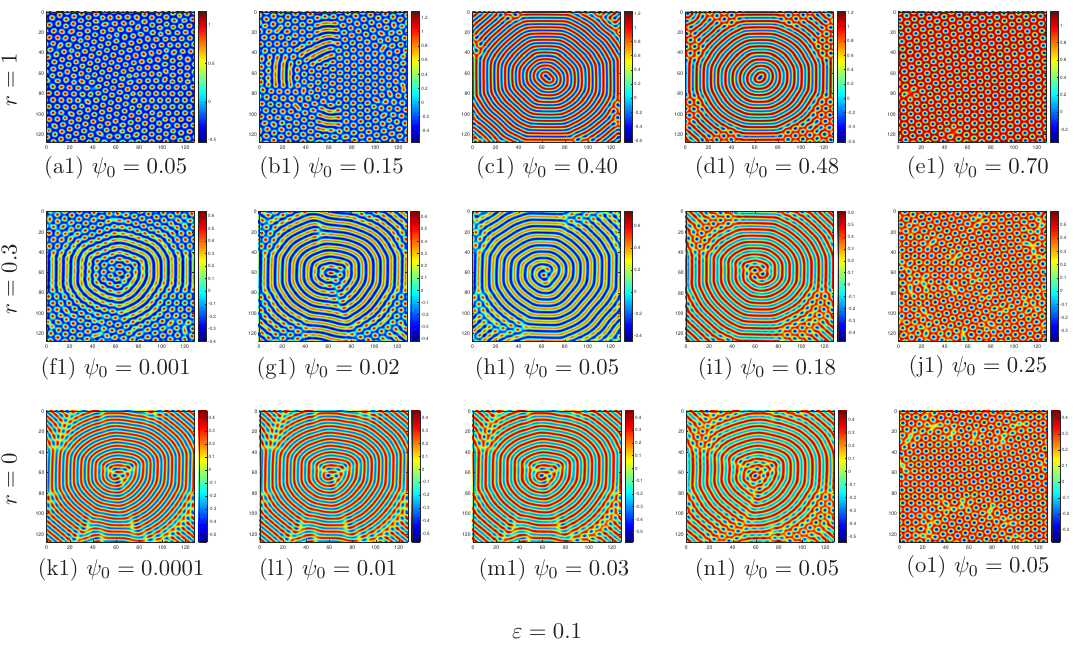}
	\caption{The evolution of the pattern formation using the four-stage third-order IMEX-RK method \eqref{eq2.8} in two dimension with $\epsilon=0.10$, $a=0.001$. $\tau=0.1$ and various of $r$ and $\psi_0$ at $T=1000$.}
	\label{figure5.8.1}
\end{figure}
\begin{figure}[htb]
	\centering
	\includegraphics[width=0.92\textwidth]{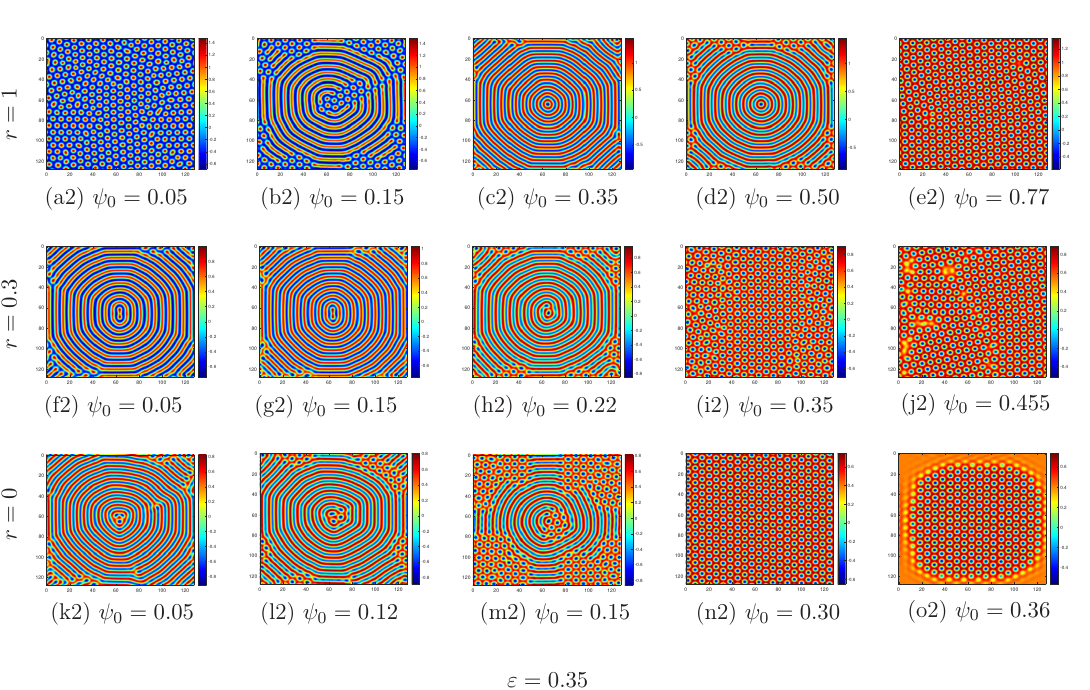}
	\caption{The evolution of the pattern formation using the four-stage third-order IMEX-RK method \eqref{eq2.8} in two dimension with $\epsilon=0.35$, $a=0.001$. $\tau=0.1$ and various of $r$ and $\psi_0$ at $T=1000$.}
	\label{figure5.8.2}
\end{figure}

The pattern formations with $\epsilon=0.1$ and $\epsilon=0.35$, for diverse $r$ and $\psi_0$, are displayed in Figures \ref{figure5.8.1} and \ref{figure5.8.2} respectively. Substantially, with the increase of $\psi_0$, the 
patterns change from stripes to stripes-triangular and triangular, which is consistent with the growth results in \cite{elder2004modeling,elder2002modeling}. Specifically, there exists dots prior to stripes with larger $r$, this also makes the form of steady-state triangular needs larger $\psi_0$. Similarly, triangular-liquid coexistence regions show up easier for lower $r$ when $\psi_0$ is relatively large.
\subsection{3D phase transition behaviors}
In this example, we simulate the results of phase transition behaviors in a cube domain $\Omega=\left[0,32\right]^3$, the parameters are set as $a=0.01$, $\epsilon=0.25$, $\tau=0.1$ and $T=1500$, with $128\time128\time128$ uniform spatial mesh. Parallel to 2D simulation, we define initial condition by
\begin{align*}
	\phi_0=0.285+0.1\,\mathrm{rand}(x,y,z),
\end{align*}
From Figure \ref{figure5.9}, we can see that the simulated dynamic is a process from disorder to order. We show the isosurface of $0.285$ of the numerical solution in the first row of Figure \ref{figure5.9}, and then the slice and surface in the next two rows, respectively.
\begin{figure}[htbp]
	\centering
	\subfigure
	{
		\begin{minipage}[b]{.16\linewidth}
			\centering
			\includegraphics[width=1.4\textwidth]{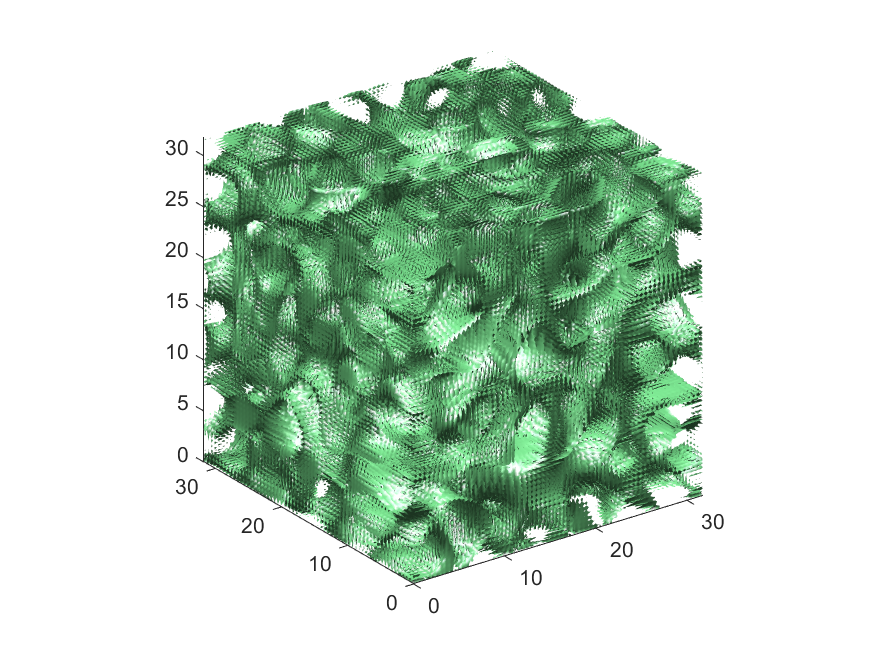}
		\end{minipage}
	}
	\subfigure
	{
		\begin{minipage}[b]{.16\linewidth}
			\centering
			\includegraphics[width=1.4\textwidth]{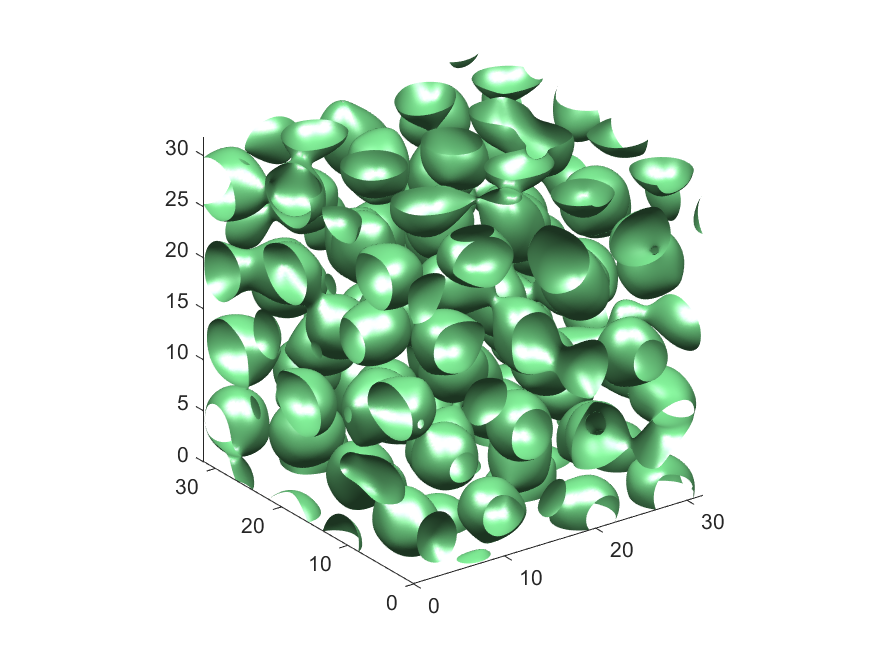}
		\end{minipage}
	}
	\subfigure
	{
		\begin{minipage}[b]{.16\linewidth}
			\centering
			\includegraphics[width=1.4\textwidth]{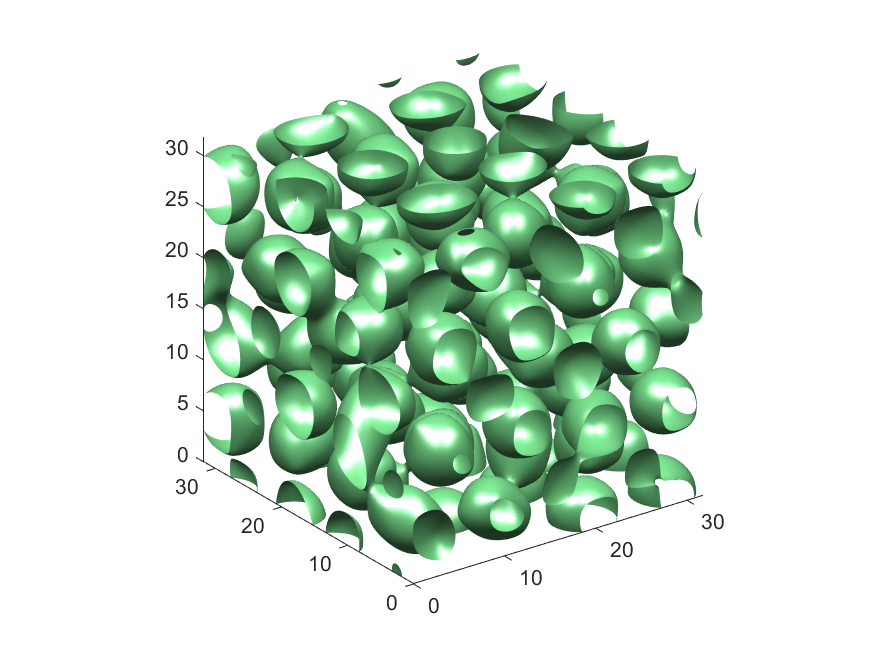}
		\end{minipage}
	}
	\subfigure
	{
		\begin{minipage}[b]{.16\linewidth}
			\centering
			\includegraphics[width=1.4\textwidth]{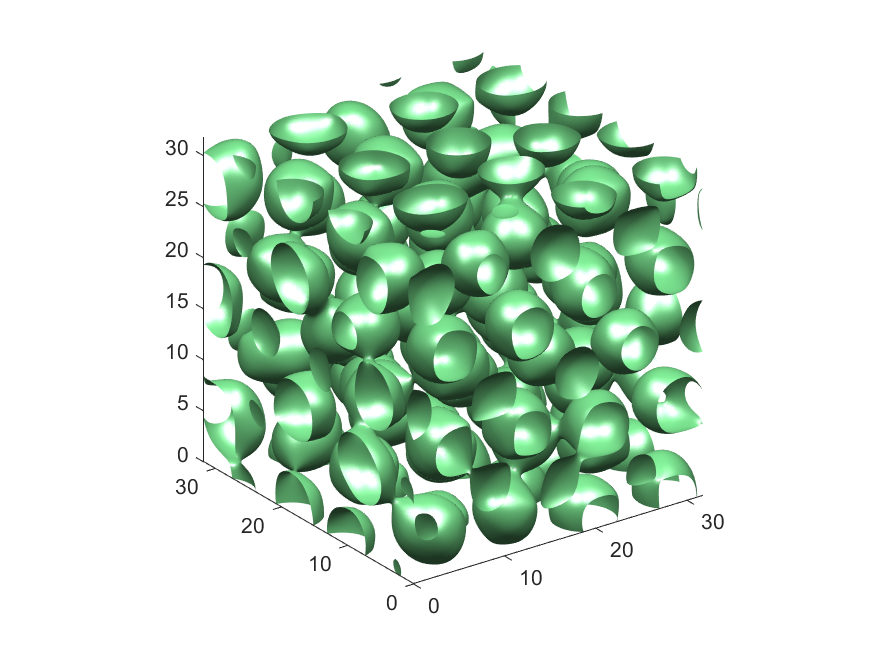}
		\end{minipage}
	}
	
	\subfigure
	{
		\begin{minipage}[b]{.16\linewidth}
			\centering
			\includegraphics[width=1.4\textwidth]{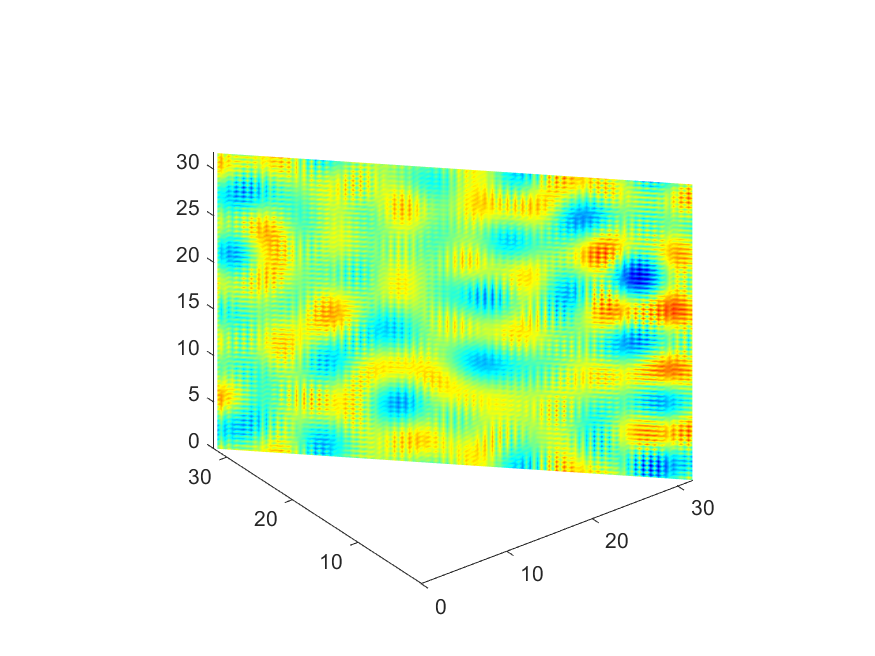}
		\end{minipage}
	}
	\subfigure
	{
		\begin{minipage}[b]{.16\linewidth}
			\centering
			\includegraphics[width=1.4\textwidth]{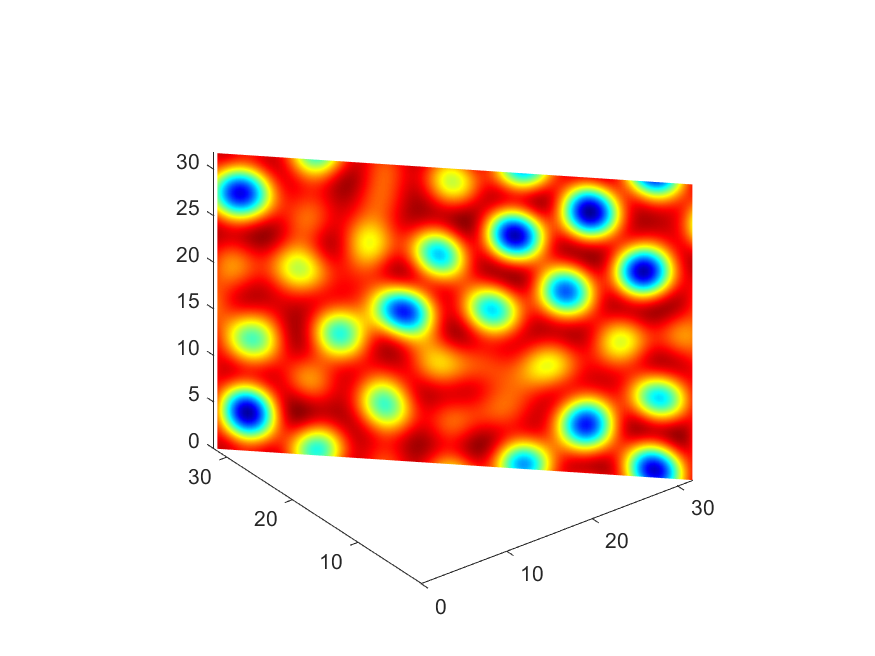}
		\end{minipage}
	}
	\subfigure
	{
		\begin{minipage}[b]{.16\linewidth}
			\centering
			\includegraphics[width=1.4\textwidth]{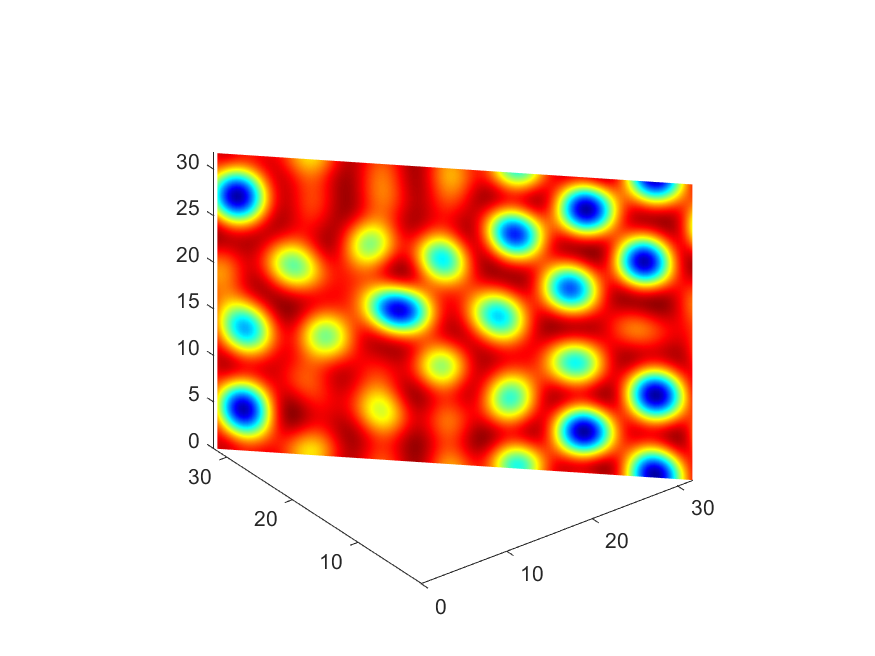}
		\end{minipage}
	}
	\subfigure
	{
		\begin{minipage}[b]{.16\linewidth}
			\centering
			\includegraphics[width=1.4\textwidth]{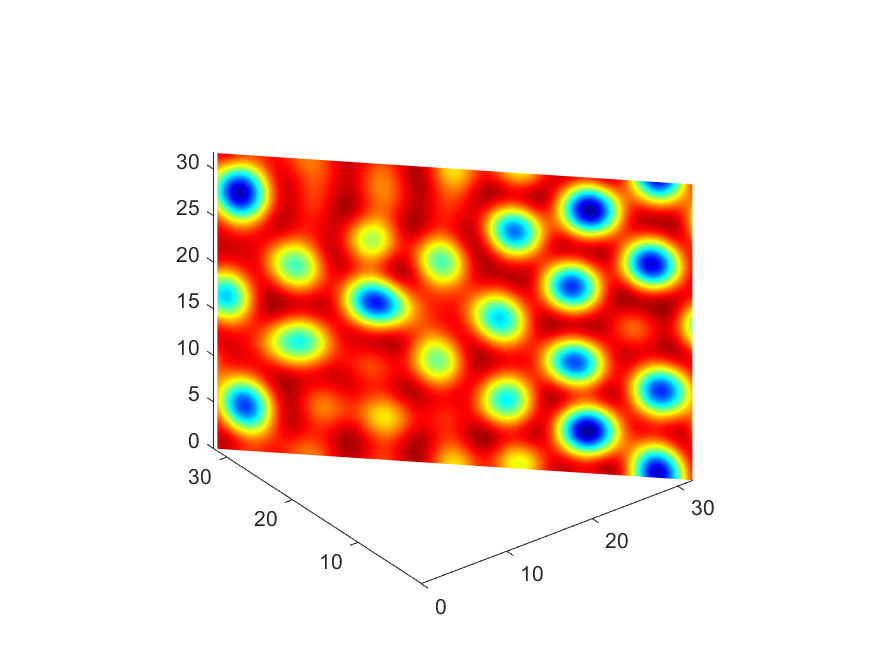}
		\end{minipage}
	}
	
	\subfigure[$t=450$]
	{
		\begin{minipage}[b]{.16\linewidth}
			\centering
			\includegraphics[width=1.4\textwidth]{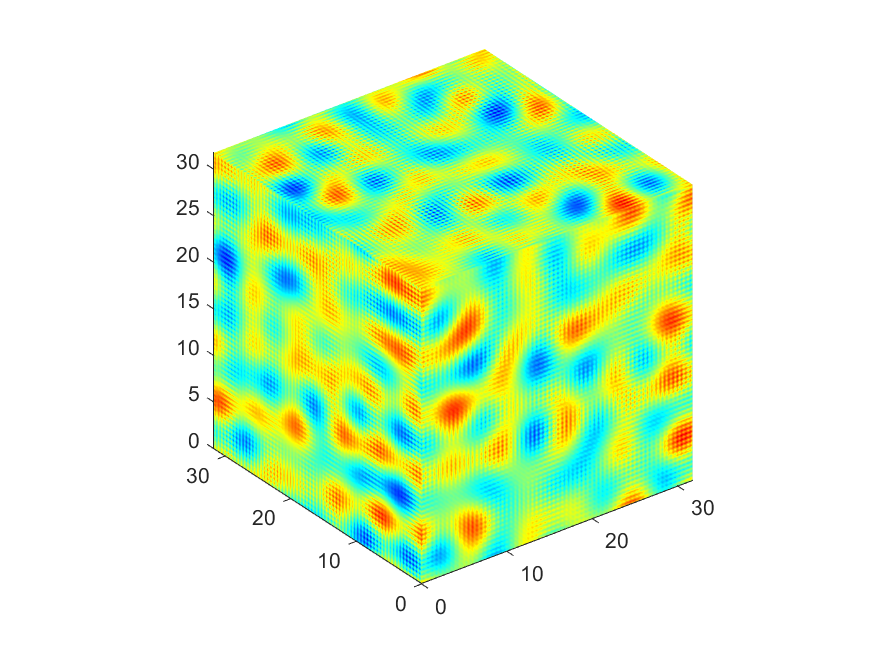}
		\end{minipage}
	}
	\subfigure[$t=700$]
	{
		\begin{minipage}[b]{.16\linewidth}
			\centering
			\includegraphics[width=1.4\textwidth]{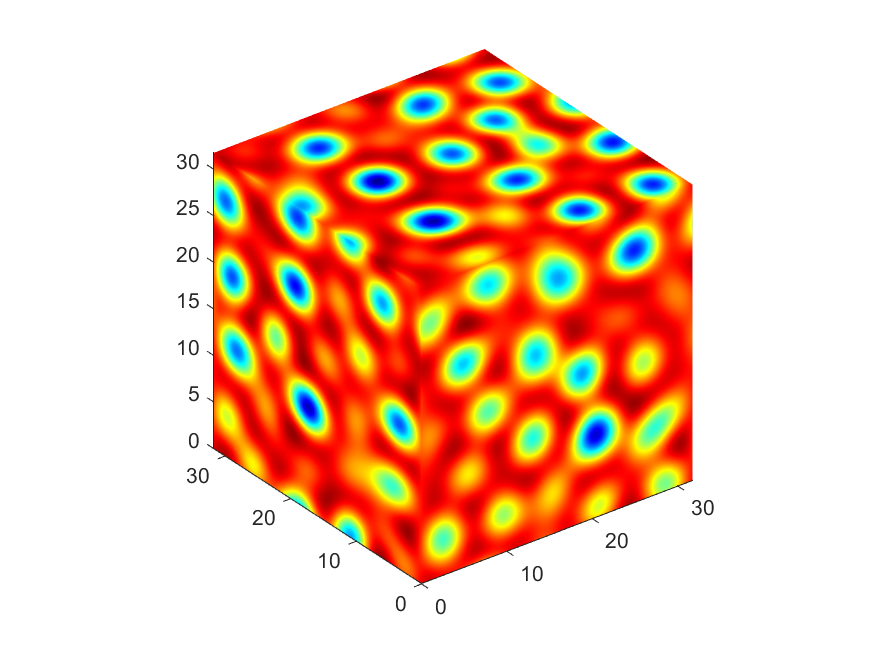}
		\end{minipage}
	}
	\subfigure[$t=900$]
	{
		\begin{minipage}[b]{.16\linewidth}
			\centering
			\includegraphics[width=1.4\textwidth]{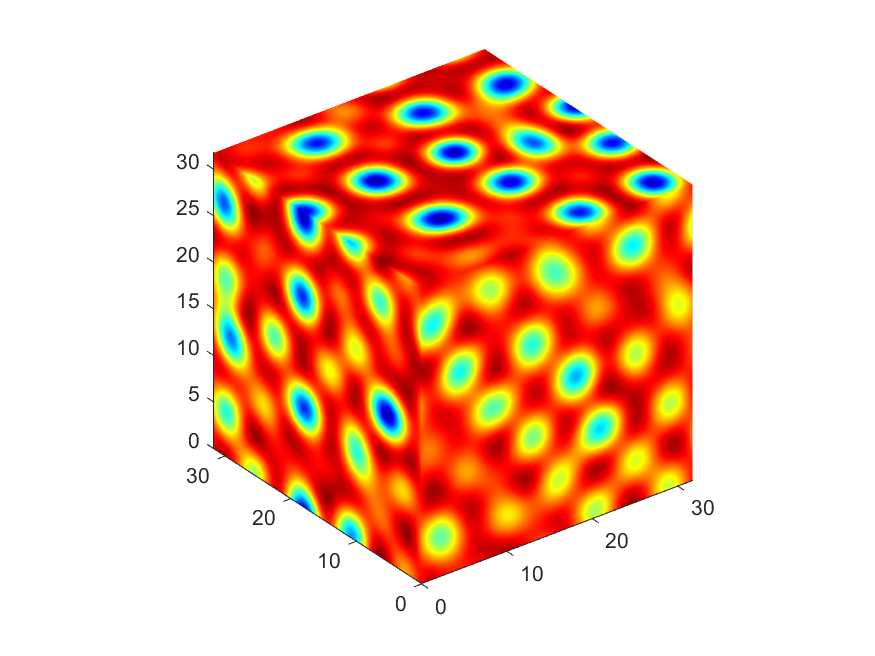}
		\end{minipage}
	}
	\subfigure[$t=1500$]
	{
		\begin{minipage}[b]{.16\linewidth}
			\centering
			\includegraphics[width=1.4\textwidth]{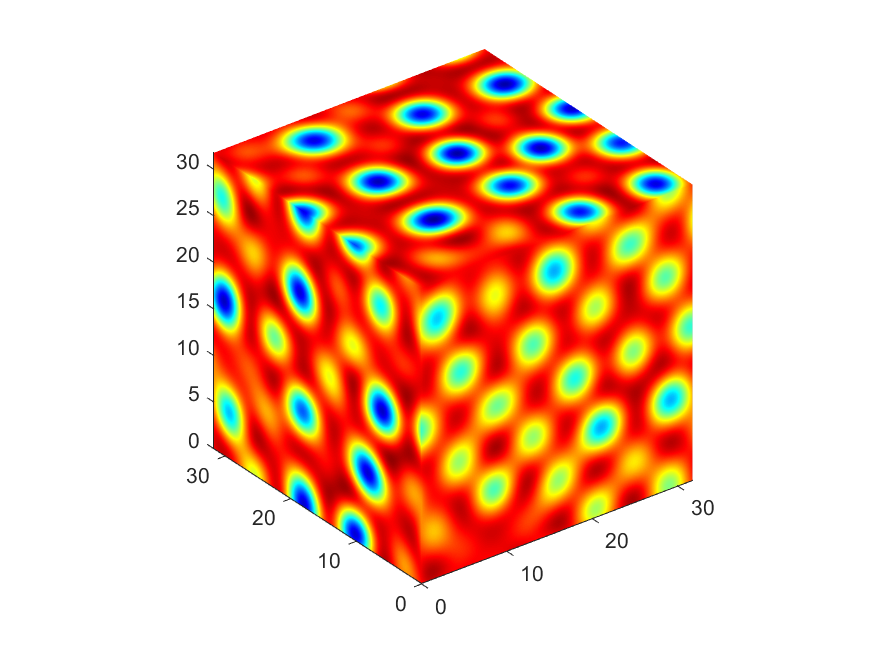}
		\end{minipage}
	}
	\caption{The evolution of 3D transition behaviors using the four-stage third-order IMEX-RK method \eqref{eq2.8} with $\epsilon=0.25$, $a=0.01$, $\tau=0.1$ at $t=450,700,900,1500$, respectively.}
	\label{figure5.9}
\end{figure}

\section{Conclusion}
In this work, we have demonstrated that the constructed IMEX-RK methods achieve 
$L^{\infty}$-convergence and uniform boundedness for the PFC model by introducing an auxiliary problem that effectively eliminates the Lipschitz nonlinearity commonly encountered in gradient flows. The analysis leverages Cauchy's interface theorem at the matrix level, facilitated by the uniform boundedness of the intermediate numerical solutions. Specifically, within the auxiliary problem, the fourth-order nonlinear term is truncated to exhibit quadratic growth, thereby inherently satisfying Lipschitz continuity. Moreover, we establish an unconditional energy dissipation result for the auxiliary problem. Subsequently, we prove that the auxiliary problem and the original PFC model yield identical numerical solutions and give error analysis.

Two promising directions for future research are identified: First, although stabilizer techniques are effective, they tend to increase numerical errors, indicating that exponential time differencing (ETD) schemes may serve as more accurate alternatives. Notably, there have been recent results demonstrating the application of ETD methods to gradient flows \cite{fuyang22,fu2025higher}. Second, extending the current framework to encompass other classes of gradient flows \cite{doi:10.1142/S0218202524500441} represents a promising avenue for further investigation.

\section*{Acknowledgments}
The work of X. Li is partially supported by the National Natural Science Foundation of China (Grant Nos. 12271302, 12131014) and Shandong Provincial Natural Science Foundation for Outstanding Youth Scholar (Grant No. ZR2024JQ030). The work of J. Yang is partially supported by the National Science Foundation
of China (No. 12271240, 12426312), the fund of the Guangdong Provincial Key Laboratory of Computational Science and Material Design, China (No. 2019B030301001), and the Shenzhen NaturalScience Fund (RCJC20210609103819018).

\bibliography{PFC-IMEX}
\end{document}